\documentclass{amsart}
\calclayout
 \usepackage{amssymb,amsmath,amsfonts,epsfig,latexsym,tikz}
 \usepackage[alphabetic]{amsrefs}
 \usepackage{tikz-cd}
\usepackage{hyperref}
\usepackage{enumerate}
\usepackage{mathtools}
\usepackage{verbatim}
\usepackage{cleveref}
\usepackage[shortlabels]{enumitem}
\usepackage{subcaption}
\usepackage{array}
\usepackage{makecell}

\usetikzlibrary{positioning}
\usetikzlibrary{matrix}
\usetikzlibrary{decorations}
\usetikzlibrary{decorations.pathreplacing, decorations.pathmorphing, angles,quotes}
 
 \newtheorem{theorem}{Theorem}[section]
\newtheorem{proposition}[theorem]{Proposition}
\newtheorem{lemma}[theorem]{Lemma}
\newtheorem{corollary}[theorem]{Corollary}
\newtheorem{conjecture}[theorem]{Conjecture}
\newtheorem{claim}[theorem]{Claim}

\newcommand{\F}{\mathbf{k}}

\newcommand{\Span}{\operatorname{span}}
\newcommand{\Lam}{\mathsf{Lam}}

\theoremstyle{definition}
\newtheorem{definition}[theorem]{Definition}
\newtheorem{remark}[theorem]{Remark}
\newtheorem{example}[theorem]{Example}

\begin{document}

\title[Rigidity matroids and linear algebraic matroids]{Rigidity matroids and linear algebraic matroids with applications to matrix completion and tensor codes}
\author{Joshua Brakensiek}\email{josh.brakensiek@berkeley.edu}
\author{Manik Dhar}\email{dmanik@mit.edu}
\author{Jiyang Gao}\email{jgao@math.harvard.edu}
\author{Sivakanth Gopi}\email{sivakanth@openai.com}
\author{Matt Larson}
\email{mattlarson@princeton.edu}

\begin{abstract}
We establish a connection between problems studied in rigidity theory and matroids arising from linear algebraic constructions like tensor products and symmetric products. A special case of this correspondence identifies the problem of giving a description of the correctable erasure patterns in a maximally recoverable tensor code with the problem of describing bipartite rigid graphs or low-rank completable matrix patterns. Additionally, we relate dependencies among symmetric products of generic vectors to graph rigidity and symmetric matrix completion. With an eye toward applications to computer science, we study the dependency of these matroids on the characteristic by giving new combinatorial descriptions in several cases, including the first description of the correctable patterns in an $(m, n, a=2, b=2)$ maximally recoverable tensor code.
\end{abstract}
 
\maketitle

\vspace{-20 pt}

\section{Introduction}

Given a graph $G$, the graph rigidity problem in $\mathbb{R}^d$ asks whether a generic embedding of the vertices of $G$ into $\mathbb{R}^d$ is \emph{rigid}, i.e., whether every motion of $G$ which preserves the lengths of edges comes from a rigid motion of $\mathbb{R}^d$. This problem has been studied since the time of Maxwell \cite{Maxwell}. When $d=2$, the Pollaczek-Geiringer--Laman theorem \cites{PollaczekGeiringer,Laman} gives a simple characterization of rigid graphs. There is no known generalization to embeddings of graphs into $\mathbb{R}^d$ for any $d \ge 3$.

The rigid graphs on $n$ vertices form the spanning sets of a \emph{matroid} on the ground set $\binom{[n]}{2}$, the set of edges of the complete graph with vertex set $[n] = \{1, \dotsc, n\}$. The language of matroids gives a convenient framework and powerful tools for analyzing rigidity \cites{JacksonRigidity,CJT2,Graver,AbstractRigidity}. We refer to \cite{Oxley} for undefined matroid terminology. 

Several other matroids related to rigidity have been introduced, such as Kalai's \emph{hyperconnectivity matroid} \cite{KalaiHyperconnectivity} $\mathcal{H}_n(d)$, for $0 \le d \le n$. This is a matroid whose ground set is $\binom{[n]}{2}$ which has similar formal properties to the usual graph rigidity matroid. Kalai used the hyperconnectivity matroid to show the existence of highly connected subgraphs in graphs with a large number of edges. Hyperconnectivity was used to study polytopal realizations of certain simplicial spheres called \emph{higher associahedra} \cites{RuizSantos1,RuizSantos2,RuizBipartite}.

More recently introduced is the \emph{bipartite rigidity matroid} $\mathcal{B}_{m,n}(a,b)$ of Kalai, Nevo, and Novak \cite{kalai2016bipartite}. This is a matroid on $[m] \times [n]$ which gives a version of rigidity for an embedding of a bipartite graph, where the parts have size $m$ and $n$, into $\mathbb{R}^{a} \oplus \mathbb{R}^b$ in such a way that it respects the direct sum structure. When $a = b$, the restriction of the hyperconnectivity matroid to the set of bipartite graphs coincides with the bipartite rigidity matroid. 
See Section~\ref{ssec:rigiditydef} for precise definitions of these matroids. 

The rigidity matroids mentioned above are closely related to \emph{low-rank matrix completion matroids}. Given $d$ and $n$, the \emph{symmetric matrix completion matroid} $\mathcal{S}_n(d)$ is a matroid on ground set $\binom{[n]}{2} \sqcup [n]$. A subset $S$ is independent if a matrix where the entries corresponding to $S$ have been filled in with generic complex numbers can be completed to a symmetric matrix of rank at most $d$. 
This is an algebraic matroid realized by the variety of $n \times n$ symmetric matrices of rank at most $d$. By \cite{GrossSullivant}*{Theorem 2.4}, the matroid describing graph rigidity in $\mathbb{R}^{d-1}$ is obtained by contracting the elements corresponding to the diagonal in $\mathcal{S}_n(d)$. In particular, a description of $\mathcal{S}_n(d)$ gives a description of graph rigidity in $\mathbb{R}^{d-1}$. The symmetric matrix completion matroid has been studied in connection with maximum likelihood problems in algebraic statistics \cites{BlekhermanSinn,BernsteinBlekhermanLee}. When $d=2$, it was studied from the perspective of tropical geometry in \cite{CLY}.

Similarly, there is a \emph{matrix completion matroid} which describes when an $m \times n$ matrix which has been partially filled in with generic complex numbers can be completed to a matrix of rank at most $d$. The matrix completion matroid was studied in \cites{bernstein2017completion,Tsakiris}. This matroid is equal to the bipartite rigidity matroid $\mathcal{B}_{m,n}(d,d)$ \cite{SingerCucuringu}*{Section 4}. As skew-symmetric matrices have even rank, one can only consider rank $d$ skew-symmetric matrix completion when $d$ is even. In this case, the analogously defined \emph{skew-symmetric matrix completion matroid} coincides with $\mathcal{H}_n(d)$ \cite{CrespoRuizSantos}*{Proposition 3.1}.

\medskip
We will establish a connection between the above rigidity matroids and matroids arising from natural linear algebraic constructions. Suppose that we have $n$ generic vectors $v_1, \dotsc, v_n$ in an $r$-dimensional vector space $V$ over an infinite field of characteristic $p \ge 0$. We assume $n \ge r$, so the vectors span $V$. Then we obtain $\binom{n+1}{2}$ vectors $v_1^2, v_1v_2, \dotsc, v_{n-1}v_n, v_n^2$ in $\operatorname{Sym}^2 V$. These vectors represent a matroid $\mathrm{S}_n(r, p)$ on the ground set $\binom{[n]}{2} \sqcup [n]$. This matroid depends only on the characteristic of the field, and in particular does not depend on the choice of vectors (provided they are sufficiently generic). We call this the \emph{symmetric power matroid}. See Section~\ref{ssec:linearalg} for a more detailed definition. 

Similarly, we obtain $\binom{n}{2}$ vectors $v_1 \wedge v_2, \dotsc, v_{n-1} \wedge v_n$ in $\wedge^2V$. These vectors represent a matroid $\mathrm{W}_n(r, p)$, which we call the \emph{wedge power matroid}. If we also have $m$ generic vectors $u_1, \dotsc, u_m$ in an $s$-dimensional vector space $U$ over the same field, for some $m \ge s$, then we obtain $mn$ vectors $u_1 \otimes v_1, \dotsc, u_m \otimes v_n$ in $U \otimes V$. These vectors represent a matroid $\mathrm{T}_{m,n}(s,r, p)$, which we call the \emph{tensor matroid}. As with the symmetric power matroids, these matroids depend only on the characteristic of the field. 

Similar constructions for arbitrary matroids have been studied classically in the matroid literature \cites{Lovasz,Mason,las1981products}. The case of symmetric powers has attracted particular attention recently in connection with applications to tropical geometry \cites{DraismaRincon,Anderson}.

\begin{table}[ht]
    \centering
    \begin{tabular}{|c||c|c|c|}
         \hline
         & Bipartite & Symmetric & Skew-symmetric \\
         \hline\hline
         Rigidity matroids & \makecell{Bipartite rigidity matroid \\ $\mathcal{B}_{m,n}(a,b)$} & \makecell{Graph rigidity matroid \\ $\mathcal{S}_n(d+1)/\{\text{diags}\}$} & \makecell{Hyperconnectivity matroid \\ $\mathcal{H}_n(d)$}\\
         \hline
         \makecell{Rank completion \\ matroids} & \makecell{Matrix completion matroid \\ $\mathcal{B}_{m,n}(d,d)$} & \makecell{Sym. matrix completion \\ $\mathcal{S}_n(d)$} & \makecell{Skew-sym. matrix completion \\ $\mathcal{H}_n(d)$, $d$ is even}\\
         \hline
         \makecell{Linear algebraic \\ matroids} & \makecell{Tensor matroid \\ $\mathrm{T}_{m,n}(s,r,p)$} & \makecell{Sym. power matroid \\ $\mathrm{S}_n(r,p)$} & \makecell{Wedge power matroid \\ $\mathrm{W}_n(r,p)$}\\
         \hline
    \end{tabular}
    \caption{A table of various matroids that have been mentioned. The matroids in the same column are related by Theorem~\ref{thm:equivalence}.}
    \label{tab:matroid-table}
\end{table}

\subsubsection*{Connections to Information Theory.} The problem of understanding $\mathrm{T}_{m,n}(s,r,p)$ 
has been studied extensively in information theory in connection with \emph{tensor codes} (c.f., e.g., \cite{MLRH14} for a practical implementation). Consider a collection of servers arranged in an $m \times n$ grid. We view the data stored on each server as an element of some (large) finite field $\textbf{k}$ of characteristic $p$. To ensure redundancy in this cluster, we constrain that each column lies in some fixed subspace of $\textbf{k}^m$ of dimension $s$ and each row lies in some fixed subspace of $\textbf{k}^n$ of dimension $r$.\footnote{More commonly, information theorists care about the codimension of these spaces, often denoted by $a = m-s$ and $b = n-r$, respectively, as that is a measure of the redundancy of the encoding.} The goal of this redundancy is that if a (small) subset of the servers fail (also called \emph{erasures}), we can completely recover the data using these various subspace constraints. The matroid of failures which do not lead to data loss is sensitive to the choice of subspaces, but if everything is chosen generically, the ``recoverability matroid'' is the matroid dual of $\mathrm{T}_{m,n}(s,r,p)$.

Gopalan et al. \cite{gopalan2017maximally} coined the term $(m,n,a,b)$-\emph{maximally recoverable tensor codes} for realizations of $\mathrm{T}_{m,n}(m-a,n-b,p)$ as tensor products of vector spaces over finite fields. Gopalan et al. gave an exponential-sized description of the spanning sets of $\mathrm{T}_{m,n}(s,r,p)$ when $m-s=1$ and conjectured a description in general. This conjecture was partially confirmed \cite{Shivakrishna_MRproduct} but was later refuted in general \cite{holzbaur2021correctable}. Overall, the focus of the information theory community has been on constructing maximally recoverable tensor codes over small fields (e.g., \cites{kong2021new,holzbaur2021correctable,roth2022higher,shivakrishna2022properties,brakensiek2023improved,athi2023structure}), although exponential-sized lower bounds are now known \cites{brakensiek2023improved,alrabiah2023ag}.
For applications, it is important to understand the matroid realized by the tensor products of vectors which are not completely generic. For example, \cite[Section 1.3]{brakensiek2023generalized} asks if the tensor products of generic vectors on the moment curve give a maximally recoverable tensor code. 

Within information theory, the study of \emph{higher order maximum-distance separable (MDS) codes} \cites{bgm2022mds,roth2022higher,bgm2023generic} has shown that, in the $m-s=1$ case, essentially the same matroid arises in many other problems. Such scenarios include designing optimally list-decodable codes \cite{shangguan2023generalized} and codes realizing particular zero patterns  \cites{dau2014gmmds,yildiz2019gmmds,lovett2021sparse,yildiz2019gmmds,liu2023linearized,brakensiek2023generalized}. Further, these equivalences imply that a construction of any one of these types of codes can be converted into the other types \cite{bgm2023generic}, and they have led to many new constructions and analyses of near-optimal codes \cites{guo2023randomly,alrabiah2023randomly,alrabiah2023ag,bdgz2023,ron2024efficient}. Very recently, these connections also led to a novel proof of the Pollaczek-Geiringer–Laman theorem \cite{bell2023kapranov}.

\medskip
\noindent Our first main result is a correspondence between the above linear algebraic matroids and rigidity matroids. Recall that the dual of a matroid $\mathrm{M}$ is the matroid whose bases are the complements of the bases of $\mathrm{M}$. 

\begin{theorem}\label{thm:equivalence}
We have the following matroid dualities:
\begin{enumerate}
	\item\label{symcase} The symmetric power matroid $\mathrm{S}_n(n-d, 0)$ is dual to the symmetric matrix completion matroid $\mathcal{S}_n(d)$. 
	\item\label{skewcase} The wedge power matroid $\mathrm{W}_n(n-d, 0)$ is dual to the hyperconnectivity matroid $\mathcal{H}_n(d)$.
	\item \label{bipartitecase} The tensor matroid $\mathrm{T}_{m,n}(m-a, n-b, 0)$ is dual to the bipartite rigidity matroid $\mathcal{B}_{m,n}(a,b)$. 
\end{enumerate}
\end{theorem}

As a corollary, the correctable erasure patterns in a $m \times n$ maximally recoverable tensor code with $a$ column and $b$ row parity checks over a field of sufficiently large characteristic are precisely the independent sets in the bipartite rigidity matroid $\mathcal{B}_{m,n}(a,b)$. 

Motivated by applications to information theory, we use Theorem~\ref{thm:equivalence}(\ref{bipartitecase}) to study bipartite rigidity. This strategy is well-suited to understanding the case when $m - a$ is small. We give exact characterizations of the independent sets in $\mathrm{T}_{m,n}(s,r,p)$ for all $p$ when $s \le 3$, see Section~\ref{sec:m-asmall}. When $m -a \le 3$, this gives a characterization of the spanning sets in $\mathcal{B}_{m,n}(a,b)$. In particular, we are able to prove the following theorem. 

\begin{theorem}\label{thm:smallm-a}
If $m-a \le 3$, then there is a deterministic algorithm to compute the rank function of $\mathcal{B}_{m,n}(a,b)$ which runs in time polynomial in $m + n$. 
\end{theorem}

Note that there are a few polynomial time algorithms to compute the rank function of $\mathcal{B}_{m,n}(a,b)$ when $a = 1$, including ones based on the maximum flow problem \cite{bgm2022mds}, total dual integral programs \cite{bgm2023generic}, and invariant theory \cite{bgm2023generic}. See also \cite{whiteley1989matroid}.

The rank of $\mathcal{B}_{m,n}(a,b)$ is $an + bm - ab$. In particular, any graph with more than $an + bm - ab$ edges must be dependent in $\mathcal{B}_{m,n}(a,b)$. Furthermore, for each subset $S \subseteq [m]$ and $T \subseteq [n]$, if we restrict $\mathcal{B}_{m,n}(a,b)$ to the edges $S \times T$, then we obtain $\mathcal{B}_{|S|, |T|}(a, b)$. See \cite{kalai2016bipartite}*{Lemma 3.7}.
In particular, if $G$ is independent in $\mathcal{B}_{m,n}(a,b)$ and $|S| \ge a$, $|T| \ge b$, then $G$ must have at most $|T|a + |S|b - ab$ edges in $S \times T$. 
  This observation gives rise to the following family of circuits, i.e., minimal dependent sets. 

\begin{definition}
A circuit $C$ of $\mathcal{B}_{m,n}(a,b)$ is a \emph{Laman circuit} if the edges of $C$ are contained in $S \times T \subseteq [m] \times [n]$, and $C$ has more than $|T|a + |S|b - ab$ edges.
\end{definition}

The name ``Laman circuits'' comes from the Pollaczek-Geiringer--Laman theorem \cites{PollaczekGeiringer,Laman}, which shows that all circuits of the graph rigidity matroid in $\mathbb{R}^2$ are Laman circuits. Laman circuits are called ``regularity'' conditions in the information theory literature \cite{gopalan2017maximally}. These authors showed that the Laman circuits completely describe $\mathrm{T}_{m,n}(m-a, n-b, p)$ when $a=1$. For all three rigidity problems mentioned above, there are typically circuits which are not Laman circuits. Theorem~\ref{thm:smallm-a} allows us to characterize exactly when all circuits of $\mathcal{B}_{m,n}(a,b)$ are Laman circuits.

\begin{example}\label{ex:nonlaman}\cite{kalai2016bipartite}*{Example 5.5}\cite{holzbaur2021correctable}*{Lemma 4}
The subset of $[5] \times [5]$ depicted as blue $\star$ in Figure~\ref{fig:nonlaman} is a circuit of $\mathcal{B}_{5,5}(2,2)$ which is not a Laman circuit. This is most easily seen from the perspective of low-rank matrix completion. If we fill in the entries corresponding to the elements of the circuit with generic complex numbers, then there will be two conflicting conditions on the $(3,3)$ entry. The complement of this circuit is dependent in $\mathrm{T}_{5,5}(3,3,p)$ for any $p$, as two $2 \times 2$-dimensional tensors in a $3 \times 3$-dimensional space necessarily intersect.
\end{example}

\begin{figure}[ht]
\begin{center}
\begin{tikzpicture}
\newcommand{\cbb}{\Huge $\star$}
\newcommand{\cbr}{\Huge $\diamond$}

\fill[blue!40!white] (0,0) rectangle (3,2);
\fill[blue!40!white] (0,0) rectangle (2,3);
\fill[blue!40!white] (5,5) rectangle (3,2);
\fill[blue!40!white] (5,5) rectangle (2,3);

\fill[red!50!white] (0,5) rectangle (2,3);
\fill[red!50!white] (5,0) rectangle (3,2);

\draw (0.5,0.5) node {\cbb};
\draw (0.5,1.5) node {\cbb};
\draw (0.5,2.5) node {\cbb};
\draw (1.5,0.5) node {\cbb};
\draw (1.5,1.5) node {\cbb};
\draw (1.5,2.5) node {\cbb};
\draw (2.5,0.5) node {\cbb};
\draw (2.5,1.5) node {\cbb};

\draw (4.5,4.5) node {\cbb};
\draw (4.5,3.5) node {\cbb};
\draw (4.5,2.5) node {\cbb};
\draw (3.5,4.5) node {\cbb};
\draw (3.5,3.5) node {\cbb};
\draw (3.5,2.5) node {\cbb};
\draw (2.5,4.5) node {\cbb};
\draw (2.5,3.5) node {\cbb};

\draw (4.5,0.5) node {\cbr};
\draw (4.5,1.5) node {\cbr};
\draw (3.5,0.5) node {\cbr};
\draw (3.5,1.5) node {\cbr};

\draw (0.5,4.5) node {\cbr};
\draw (0.5,3.5) node {\cbr};
\draw (1.5,4.5) node {\cbr};
\draw (1.5,3.5) node {\cbr};

\draw (0,0) grid (5,5);
\end{tikzpicture}
\end{center}
\caption{A circuit of $\mathcal{B}_{5,5}(2,2)$ which is not a Laman circuit. The squares of the $5 \times 5$ grid represent the ground set of the matroid, with the blue $\star$ squares representing the circuit elements. The red $\diamond$ squares form the corresponding circuit in $\mathrm{T}_{5,5}(3,3,p)$ for any $p$.}\label{fig:nonlaman}
\end{figure}

\begin{corollary}\label{cor:laman}
All circuits in $\mathcal{B}_{m,n}(a,b)$ are Laman circuits if and only if at least one of the following holds: (1) $a \le 1$, (2) $b \le 1$, (3) $m-a \le 2$, or (4) $n - b \le 2$. 
\end{corollary}

The case when $a \in\{0, m\}$ or  $b \in \{0, n\}$ is trivial, and the case when $a=1$ or $b=1$ was proven in \cite{whiteley1989matroid} and independently in \cite{gopalan2017maximally}. 
Note that we are able to describe $\mathcal{B}_{m,n}(a,b)$ when $m-a =3$ even though the Laman condition usually fails. To do this, we find an additional family of combinatorial inequalities which rule out circuits such as the one in Example~\ref{ex:nonlaman}. 

\medskip

Applications to information theory make use of $\mathrm{T}_{m,n}(s,r,p)$ when $p > 0$. While the bases of $\mathrm{T}_{m,n}(s,r,p)$ are bases of $\mathrm{T}_{m,n}(s,r,0)$ (Proposition~\ref{prop:uppersemi}), it is not obvious that these matroids are equal. Indeed, \cite{Anderson}*{Example 2.18} shows that $\mathrm{S}_4(2, 2) \not= \mathrm{S}_4(2,0)$, so the symmetric power matroid depends on the characteristic. Some of the proofs of the Pollaczek-Geiringer--Laman theorem, such as the ones in \cites{LovaszYemini,bell2023kapranov}, show that the rank of the matrices considered in $2$-dimensional graph rigidity do not depend on the characteristic. 
We show that the tensor matroid does not depend on the characteristic in several cases.

\begin{theorem}\label{thm:char}
If $s \le 3,$ $m - s \le 1$, or $m - s = n - r = 2$, then $\mathrm{T}_{m,n}(s,r,p) = \mathrm{T}_{m,n}(s,r,0)$. 
\end{theorem}

One of the deepest results on bipartite rigidity is the following theorem of Bernstein. The proof crucially uses the interpretation of bipartite rigidity in terms of low-rank matrix completion to reformulate the problem in terms of tropical geometry. Bernstein then uses several ingenious ideas to obtain the following combinatorial characterization of the independent sets of the $(2, 2)$ bipartite rigidity matroid. 

\begin{theorem}\cite{bernstein2017completion}\label{thm:Bernstein}
A bipartite graph is independent in $\mathcal{B}_{m,n}(2,2)$ if and only if it has an edge orientation with no directed cycles or alternating cycles. 
\end{theorem}

That is, $G$ is independent in $\mathcal{B}_{m,n}(2,2)$ if and only if $G$ has an acyclic orientation such that no cycle of $G$ is oriented so that the edges alternate in orientation. 
As part of our proof of Theorem~\ref{thm:char}, we give an elementary proof (\Cref{prop:char-free-Bernstein}) of the sufficiency part of Theorem~\ref{thm:Bernstein}, i.e., if a bipartite graph  $G$ has an edge orientation with no directed cycles or alternating cycles, then it is independent in $\mathcal{B}_{m,n}(2,2)$. Unlike Bernstein's original proof, our argument establishes the stronger statement that $G$ is independent in the dual of $\mathrm{T}_{m,n}(m-2, n-2,p)$ for any $p$.  Together with Proposition~\ref{prop:uppersemi}, this proves the independence of the characteristic in this case. 

As a consequence, we have a combinatorial description of correctable patterns in an $(m, n, a=2, b=2)$ maximally recoverable tensor code.
\begin{proposition}\label{cor:it-Bernstein}
Let $C \subseteq \F^{m \times n}$ be an $(m,n,a=2,b=2)$ maximally recoverable tensor code. An erasure pattern $E \subseteq C$ is correctable if and only if $E$, when viewed as a bipartite graph, has an edge orientation with no directed cycles or alternating cycles.
\end{proposition}

Despite the combinatorial nature of the description in Theorem~\ref{thm:Bernstein}, we do not know a polynomial time algorithm to check independence in $\mathcal{B}_{m,n}(2,2)$. We do not even know a coNP certificate, i.e., a certificate that a graph is \emph{not} independent in $\mathcal{B}_{m,n}(2,2)$, that can be checked in polynomial time. A candidate coNP certificate for independence in $\mathcal{B}_{m,n}(a,b)$ is given in \cite{JacksonTanigawa}*{Conjecture 6.4}. 

Finally, we give a conjectural description of the bases of $\mathcal{B}_{m,n}(d, d)$ for all $d$, generalizing Theorem~\ref{thm:Bernstein} (Conjecture~\ref{conj:generalizedbernstein}). Using a ``coning'' operation (Proposition~\ref{prop:cone}), this gives a description of $\mathcal{B}_{m,n}(a,b)$ for all $a, b$. 
We show that our conjecture implies that $\mathrm{T}_{m,n}(s,r,p)$ is independent of $p$.

\begin{table}[ht]
    \centering
    \begin{tabular}{|c|c|}
         \hline
         Cases & Description of $\mathrm{T}_{m,n}(s,r,p)$ \\
         \hline\hline
         $s=1$ or $r=1$ & Corollary~\ref{lem:r=1}\\
         \hline
         $s=2$ or $r=2$ & Proposition~\ref{lem:m-a=2}\\
         \hline
         $s=3$ or $r=3$ & Proposition~\ref{lem:m-a=3}\\
         \hline
         $m-s=1$ or $n-r=1$ & \makecell{\cite[Theorem 4.2]{whiteley1989matroid} \\ \cite[Theorem 3.2]{gopalan2017maximally}}\\
         \hline
         $m-s=n-r=2$ & \makecell{\cite[Theorem 4.4]{bernstein2017completion} for $p=0$ \\ Proposition~\ref{prop:char-free-Bernstein} for $p>0$}\\
         \hline
    \end{tabular}
    \caption{Currently known cases of the structure of the matroid $\mathrm{T}_{m,n}(s,r,p)$. 
    }
    \label{tab:result-table}
\end{table}

\subsubsection*{Acknowledgements}
We thank the referees for their helpful comments. 
The first author was supported by a Microsoft Research PhD Fellowship. The second author was supported by NSF grant DMS-1953807. 
The third author would like to thank Alex Postnikov and Yibo Gao for introducing him to the matrix completion problem. 
The fifth author is supported by an ARCS fellowship. 

\section{Rigidity matroids and their duals}

In this section, we recall the definitions of the rigidity matroids. We then prove Theorem~\ref{thm:equivalence}. 

\subsection{Rigidity matroids}\label{ssec:rigiditydef}

\subsubsection{Symmetric matrix completion}

For $0 \le d \le n$, the \emph{symmetric matrix completion matroid} $\mathcal{S}_n(d)$ is the algebraic matroid realized by the $\binom{n}{2} + n$ coordinate functions on the variety of $n \times n$ symmetric matrices of rank at most $d$. More precisely, let $K$ be the field of rational functions on the variety of $n \times n$ symmetric matrices of rank at most $d$ over $\mathbb{C}$. For each $i \le j$, we have a coordinate function $x_{ij} \in K$. A subset $S \subseteq \binom{[n]}{2} \sqcup [n]$ is independent in $\mathcal{S}_n(d)$ if and only if the corresponding set of coordinate functions is algebraically independent. The dimension of the variety of $n \times n$ symmetric matrices of rank at most $d$ is $nd - \binom{d}{2}$, so the rank of $\mathcal{S}_n(d)$ is $nd - \binom{d}{2}$.  

The restriction of $\mathcal{S}_n(d)$ to $\binom{[n]}{2}$ is the matroid denoted $\mathcal{I}^d_n$ in \cite{JacksonTanigawa}*{Section 6.3}, which was introduced in \cite{KalaiHyperconnectivity}*{Section 8}. If $n = m + p$, then the restriction of $\mathcal{S}_n(d)$ to $[m] \times [p] \subseteq \binom{[n]}{2}$ is $\mathcal{B}_{m,p}(d,d)$. This is most easily seen using the description of $\mathcal{B}_{m,p}(d,d)$ as a \emph{matrix completion matroid}: the projection of the variety of $n \times n$ symmetric matrices of rank at most $d$ onto the northeast $m \times p$ corner is the variety of $m \times p$ matrices of rank at most $d$.

We use a description of $\mathcal{S}_n(d)$ which was derived in \cite{KRT}*{Section 3.2}. See Section~\ref{sec:poschar} for a discussion of an analogous calculation.

\begin{proposition}\label{prop:symmetricmatrix}
Consider the $nd \times (\binom{n}{2} + n)$ matrix $J_{\mathrm{Sym}}$ over the field $\mathbb{C}(x_{ij})_{1 \le i \le n, 1 \le j \le d}$ whose rows are labeled by pairs $(i, j)$ with $1 \le i \le n, 1 \le j \le d$, and whose columns are labeled by either $\{k, \ell\} \in \binom{[n]}{2}$ or $k \in [n]$. 
In each row labeled by $(i, j)$, we have an entry of $2x_{ij}$ in the column labeled by $i \in [n]$, and we have an entry of $x_{kj}$ in the column labeled by $\{i, k\} \in \binom{[n]}{2}$. The other entries are $0$. 
A subset $S \subseteq \binom{[n]}{2} \sqcup [n]$ is independent in $\mathcal{S}_n(d)$ if and only if the columns labeled by $S$ in $J_{\mathrm{Sym}}$ are linearly independent. 
\end{proposition}

In other words, $\mathcal{S}_n(d)$ is the column matroid of $J_{\mathrm{Sym}}$. 

\subsubsection{Hyperconnectivity}\label{sssec:hyperconnectivity}

For $0 \le d \le n$, the hyperconnectivity matroid $\mathcal{H}_n(d)$ is a matroid on $\binom{[n]}{2}$ of rank $dn - \binom{d+1}{2}$.
It was defined in \cite{KalaiHyperconnectivity} in terms of \emph{algebraic shifting}. 
It can equivalently be defined as a column matroid of an explicit matrix, see Definition~\ref{def:hyperconnectivity} below.
If $n = m + p$, then the restriction of $\mathcal{H}_n(d)$ to $[m] \times [p] \subseteq \binom{[n]}{2}$ is $\mathcal{B}_{m,p}(d,d)$. 

\begin{definition}\label{def:hyperconnectivity}
Consider the $nd \times \binom{n}{2}$ matrix over the field $\mathbb{C}(x_{ij})_{1 \le i \le n, 1 \le j \le d}$ whose rows are labeled by pairs $(i, j)$ with $1 \le i \le n, 1 \le j \le d$, and whose columns are labeled by $\{k, \ell\} \in \binom{[n]}{2}$. The row corresponding to $(i, j)$ has $x_{ki}$ in the column labeled by $\{k, j\}$ if $k < j$, $-x_{ki}$ in the column labeled by $\{k, j\}$ if $k > j$, and is $0$ otherwise. 
The \emph{hyperconnectivity matroid} $\mathcal{H}_n(d)$ is the column matroid of this matrix. 
\end{definition}

\subsubsection{Bipartite rigidity}

For $0 \le a \le m$ and $0 \le b \le n$, the bipartite rigidity matroid $\mathcal{B}_{m,n}(a,b)$ is a matroid on $[m] \times [n]$ of rank $an + bm - ab$. It was introduced in \cite{kalai2016bipartite}. See \cite{kalai2016bipartite}*{Section 1.2} for a physical interpretation of bipartite rigidity in terms of embeddings of bipartite graphs on $[m] \sqcup [n]$ into $\mathbb{R}^a \oplus \mathbb{R}^b$, where $[m]$ is embedded into $\mathbb{R}^a \oplus 0$ and $[n]$ is embedded into $0 \oplus \mathbb{R}^b$. It can equivalently be defined as the column matroid of an explicit matrix \cite{kalai2016bipartite}*{Proposition 3.3}.

\begin{definition}\label{def:bipartiterigidity}
Consider the $(an + bm) \times mn$ matrix over the field $\mathbb{C}(x_{ij}, y_{k\ell})$, where $(i, j) \in [m] \times [a]$ and $(k, \ell) \in [n] \times [b]$, whose rows are labeled by elements of $[a] \times [n]$ or $[b] \times [m]$, and whose columns are labeled by elements of $[m] \times [n]$. In the row labeled by $(j, k) \in [a] \times [n]$, we have $x_{pj}$ in the column labeled by $(p, k)$ for $p \in [m]$ and $0$ in every other column. In the row labeled by $(\ell, i) \in [b] \times [m]$, we have $y_{q \ell}$ in the column labeled by $(i, q)$ for $q \in [n]$ and $0$ in every other column. The \emph{bipartite rigidity matroid} $\mathcal{B}_{m,n}(a,b)$ is the column matroid of this matrix. 
\end{definition}

\subsection{Linear algebraic matroids}\label{ssec:linearalg}
We now give a more detailed definition of the symmetric power matroid $\mathrm{S}_n(r, p)$. The tensor matroid and wedge power matroid are defined similarly. Let $K$ be a field of characteristic $p \ge 0$. We say that $n$ vectors $v_1, \dotsc, v_n$ in $K^r$ are \emph{totally generic} if their coordinates are algebraically independent. Then the vectors $v_1^2, v_1v_2, \dotsc, v_{n-1}v_n, v_n^2$ are contained in $\operatorname{Sym}^2 K^r$, and so they define a matroid, which we denote by $\mathrm{S}_n(r, p)$. Given any vectors $v_1, \dotsc, v_n$ in an $r$-dimensional vector space $V$ over a field of characteristic $p \ge 0$, we say that they are \emph{generic} if the matroid represented by the vectors $v_1^2, v_1v_2, \dotsc, v_n^2$ in $\operatorname{Sym}^2 V$ is $\mathrm{S}_n(r,p)$. If $V$ is a vector space over an infinite field or over any sufficiently large finite field, then generic vectors exist. This follows from the fact that a nonzero polynomial $f \in K[x_1, \dotsc, x_n]$ does not vanish for every value of $x_1, \dotsc, x_n$ over an infinite field or a sufficiently large finite field, applied to the product of the maximal minors witnessing that each basis of $\mathrm{S}_{n}(r,p)$ is in fact a basis. 
See \cite[Corollary IV.1.6]{Lang} for the case of infinite fields; the case of sufficiently large finite fields follows from the Schwartz--Zippel lemma \cite{Sch80,Zip79}.

\subsection{Dualities}

In this section, we prove Theorem~\ref{thm:equivalence}. For this, it is convenient to pass to the dual picture. Suppose we have a collection of $n$ vectors in a $d$-dimensional vector space $L$. If we choose a basis for $L$, then we obtain a $d \times n$ matrix $A$. A collection of vectors is independent in the matroid represented by the vector configuration if and only if the corresponding columns of $A$ are linearly independent. Whether a given set of columns is linearly independent depends only on the row span of $A$. In particular, we can replace $A$ by any matrix with the same row span (even if its rows are linearly dependent). 

Using this formulation, we can describe $\mathrm{S}_n(r, p)$ as follows. Choose an infinite field $K$ of characteristic $p \ge 0$, and choose a generic linear subspace $L \subseteq K^n$ of dimension $r$. We then obtain a subspace $\operatorname{Sym}^2 L \subseteq \operatorname{Sym}^2 K^n$. We have a canonical basis for $\operatorname{Sym}^2K^n$, given by the vectors $e_1^2, e_1e_2, \dotsc, e_n^2$, where $e_1, \dotsc, e_n$ is the standard basis of $K^n$. To compute the independent sets of $\mathrm{S}_n(r, p)$, we choose $m$ vectors in $\operatorname{Sym}^2 K^n$ whose span is $\operatorname{Sym}^2 L$ for some $m \ge \dim \operatorname{Sym}^2L$, form the corresponding $m \times \binom{n + 1}{2}$ matrix, and check which columns are linearly independent.
To compute the dual of $\mathrm{S}_n(r, p)$, we choose vectors which span $(\operatorname{Sym}^2 L)^{\perp} \subseteq \operatorname{Sym}^2 K^n$, where the orthogonal complement is taken with respect to the usual inner product on a vector space with a basis, i.e., the dot product. We can calculate $\mathrm{W}_n(r, p)$ and $\operatorname{T}_{m,n}(s,r, p)$ in a similar way. 

\begin{proof}[Proof of Theorem~\ref{thm:equivalence}(\ref{symcase})]
As $\mathrm{S}_n(n-d, 0)$ is independent of the choice of infinite field of characteristic $0$, we may work over $K = \mathbb{C}(x_{ij})_{1 \le i \le n, 1 \le j \le d}$ and choose our generic linear subspace $L \subseteq K^n$ to be the orthogonal complement of the span of the vectors $(x_{11}, \dotsc, x_{n1}), \dotsc, (x_{1d}, \dotsc, x_{nd})$. In order to compute the dual of $\mathrm{S}_n(n-d, 0)$, we find vectors which span $(\operatorname{Sym}^2L)^{\perp}$. 
There is a surjective map $L^{\perp} \otimes K^n \to (\operatorname{Sym}^2L)^{\perp} \subseteq \operatorname{Sym}^2 K^n$ which sends $v \otimes w$ to $vw$. There is a basis for $L^{\perp} \otimes K^n$ given by vectors of the form $(x_{1j}, \dotsc, x_{nj}) \otimes e_i$ for $1 \le i \le n$ and $1 \le j \le d$. We form the matrix $A$ whose rows are given by the images of these vectors in $\operatorname{Sym}^2K^n$, written in the basis $e_1^2, e_1e_2, \dotsc, e_n^2$. The dual of  $\mathrm{S}_n(n-d, 0)$ records which columns of $A$ are linearly independent. 
We note that we obtain the matrix $J_{\operatorname{Sym}}$ of Proposition~\ref{prop:symmetricmatrix} from $A$ after multiplying the columns labeled by $i \in [n]$ by $2$, proving the equivalence. 
\end{proof}

\begin{example}
We now illustrate the proof of Theorem~\ref{thm:equivalence}(\ref{symcase}) in the case $n = 4$ and $d = 2$. The matroid $\mathcal{S}_4(2)$ is the column matroid of the matrix
$$\begin{pmatrix} x_{21} & x_{31} & x_{41} & 0 & 0 & 0 & 2x_{11} & 0 & 0 & 0 \\ 
x_{11} & 0 & 0 & x_{31} & x_{41} & 0 & 0 & 2x_{21} & 0 & 0 \\ 
0 & x_{11} & 0 & x_{21} & 0 & x_{41} & 0 & 0 & 2x_{31} & 0 \\
0 & 0 & x_{11} & 0 & x_{21} & x_{31} & 0 & 0 & 0 & 2x_{41} \\ 
x_{22} & x_{32} & x_{42} & 0 & 0 & 0 & 2x_{12} & 0 & 0 & 0 \\ 
x_{12} & 0 & 0 & x_{32} & x_{42} & 0 & 0 & 2x_{22} & 0 & 0 \\ 
0 & x_{12} & 0 & x_{22} & 0 & x_{42} & 0 & 0 & 2x_{32} & 0 \\
0 & 0 & x_{12} & 0 & x_{22} & x_{32} & 0 & 0 & 0 & 2x_{42} \\ 
\end{pmatrix}.$$
Here, in order, the rows are labeled by $(1,1), (2,1), \dotsc, (4, 1), (1,2), \dotsc, (4,2)$, and the columns are labeled by $(1,2), (1,3), (1,4),(2,3), (2,4), (3,4), 1, 2, 3, 4$. If $L$ is the subspace of $K^4$ which is the orthogonal complement of the span of the vectors $(x_{11}, x_{21}, x_{31}, x_{41})$ and $(x_{12}, x_{22}, x_{32}, x_{42})$, then the subspace $(\operatorname{Sym}^2L)^{\perp}$ of $\operatorname{Sym^2} K^4$ is spanned by the $8$ vectors $(x_{11}, x_{21}, x_{31}, x_{41}) \cdot e_1$, $(x_{11}, x_{21}, x_{31}, x_{41}) \cdot e_2, \dotsc, (x_{12}, x_{22}, x_{32}, x_{42}) \cdot e_4$. Expressing these vectors in terms of the basis $e_1e_2, e_1e_3, e_1e_4, e_2e_3, e_2e_4, e_3e_4, e_1^2, e_2^2, e_3^2, e_4^2$ for $\operatorname{Sym}^2 K^4$ gives the rows of the matrix above, up to multiplying some columns by $2$. 
\end{example}

\begin{proof}[Proof of Theorem~\ref{thm:equivalence}(\ref{skewcase})]
We work over $K = \mathbb{C}(x_{ij})_{1 \le i \le n, 1 \le j \le d}$, and choose our generic linear subspace $L \subseteq K^n$ to be the orthogonal complement of the span of the vectors $(x_{11}, \dotsc, x_{n1}), \dotsc, (x_{1d}, \dotsc, x_{nd})$. There is a surjective map $L^{\perp} \otimes K^n \to (\wedge^2L)^{\perp} \subseteq \wedge^2 K^n$. Using the basis $\{(x_{1j}, \dotsc, x_{nj}) \otimes e_i\}_{1 \le i \le n,  1 \le j \le d}$ for $L^{\perp} \otimes K^n$ and the basis $\{e_i \wedge e_j : i < j\}$ for $\wedge^2 K^n$, we see that $(\wedge^2 L)^{\perp}$ is the row span of the matrix appearing in Definition~\ref{def:hyperconnectivity}. 
\end{proof}

\begin{proof}[Proof of Theorem~\ref{thm:equivalence}(\ref{bipartitecase})]
We work over $K = \mathbb{C}(x_{ij}, y_{k\ell})$, where $(i, j) \in [m] \times [a]$ and $(k, \ell) \in [n] \times [b]$. Set $L_1 \subseteq K^m$ to be the orthogonal complement to the span of the vectors $(x_{11}, \dotsc, x_{m1}), \dotsc, (x_{1a}, \dotsc, x_{ma})$, and set $L_2 \subseteq K^n$ to be the orthogonal complement to the span of the vectors $(y_{11}, \dotsc, y_{n1}), \dotsc, (y_{1b}, \dotsc, y_{nb})$. There is a surjective map $(L_1^{\perp} \otimes K^n) \oplus (K^m \otimes L_2^{\perp}) \to (L_1 \otimes L_2)^{\perp} \subseteq K^m \otimes K^n$. This implies that $(L_1 \otimes L_2)^{\perp}$ is the row span of the matrix appearing in Definition~\ref{def:bipartiterigidity}. 
\end{proof}

\begin{example}
Using Theorem~\ref{thm:equivalence}, a special case of \cite{KalaiShifting}*{Problem 3} which is given as a conjecture in \cite{CrespoRuizSantos}*{Conjecture 4.3} becomes the following. Let $G$ be a graph on $n$ vertices, and suppose that $E(G)$ is independent in $\mathrm{W}_n(r, 0)$. Then $E(G)$ is independent in $\mathrm{S}_n(r-1, 0)$. 
We checked this for $r \le 6$. 
\end{example}

\subsection{Matrix completion in positive characteristic}\label{sec:poschar}

As mentioned in the introduction, the algebraic matroid of the variety of $n \times n$ skew-symmetric matrices of rank at most $2d$ over $\mathbb{C}$ is the hyperconnectivity matroid $\mathcal{H}_n(2d)$ \cite{CrespoRuizSantos}*{Proposition 3.1}, and the algebraic matroid of $m \times n$ matrices of rank at most $d$ over $\mathbb{C}$ is the bipartite rigidity matroid $\mathcal{B}_{m,n}(d,d)$ \cite{SingerCucuringu}*{Section 4}. 
We briefly comment on the relationship between the linear algebraic matroids in characteristic $p$ and the low-rank matrix completion matroids in characteristic $p$. This section is not used in the rest of the paper and can be skipped by the uninterested reader. For the field theory facts used in this section, see \cite{Matsumura}*{Section 26}, especially Theorem 26.6. 

We first sketch how one computes the low-rank matrix completion matroid. Let $Y_d$ be the subvariety of $\mathbb{C}^{mn}$ given by $m \times n$ matrices of rank at most $d$, and let $K(Y_d)$ be the function field of $Y_d$. The matrix completion matroid encodes when the coordinate functions $z_{ij} \in K(Y_d)$ are algebraically independent. Because we are over a field of characteristic $0$, some functions $\{z_{ij}\}$ are algebraically independent if and only if their differentials $\mathrm{d}z_{ij} \in \Omega_{K(Y_d)/\mathbb{C}}$ are linearly independent in the module of differentials of $K(Y_d)$. In order to make the module of differentials $\Omega_{K(Y_d)/\mathbb{C}}$ explicit, we use that every matrix of rank at most $d$ can be written as $AB$, where $A$ is an $m \times d$ matrix and $B$ is a $d \times n$ matrix. This means that there is a surjective map $\mathbb{C}^{md + dn} \to Y_d$, which sends a coordinate function $z_{ij}$ to $\sum_{\ell = 1}^{d} x_{i\ell} y_{\ell j}$. Because we are over a field of characteristic $0$, the pullback map $\Omega_{K(Y_d)/\mathbb{C}} \otimes_{K(Y_d)} \mathbb{C}(x_{ij}, y_{k \ell}) \to \Omega_{\mathbb{C}(x_{ij}, y_{k \ell})/\mathbb{C}}$ is injective. 
There is a basis for $\Omega_{\mathbb{C}(x_{ij}, y_{k \ell})/\mathbb{C}}$ given by $\{\mathrm{d}x_{ij}, \mathrm{d}y_{k\ell}\}$. 
Then $\mathrm{d}z_{ij}$ pulls back to $\mathrm{d}(\sum_{\ell = 1}^{d} x_{i\ell} y_{\ell j}) = \sum_{\ell=1}^{d} (x_{i\ell} \mathrm{d}y_{\ell j} + y_{\ell j} \mathrm{d}x_{i \ell})$. The matrix whose columns are given by the pullbacks of the $\mathrm{d}z_{ij}$ is exactly the matrix defining $\mathcal{B}_{m,n}(d,d)$. A similar argument can be used to compute the symmetric matrix completion matroid, using that an $n \times n$ symmetric matrix of rank at most $d$ can be written as $AA^t$, where $A$ is an $n \times d$ matrix, or the skew-symmetric matrix completion matroid, using that an $n \times n$ skew-symmetric matrix of rank at most $2d$ can be written as $AB^t - BA^t$ for $A, B$ $n \times d$ matrices. 

This argument breaks down in positive characteristic due to the presence of inseparable extensions. Given a finitely generated extension of fields $K/L$, we say that $a_1, \dotsc, a_\ell$ are \emph{separably algebraically independent} if there are $b_1, \dotsc, b_{p}$ such that $a_1, \dotsc, a_{\ell}, b_1, \dotsc, b_p$ is a \emph{separating transcendence basis} for $K/L$, i.e., they are algebraically independent and $K/L(a_1, \dotsc, b_p)$ is a finite separable extension. We say that $K/L$ is \emph{separable} if it has a separating transcendence basis. 

Over a field $\mathbf{k}$ of positive characteristic, it is no longer true that we can test algebraic independence of a collection $a_1, \dotsc, a_{\ell}$ by checking the linear independence of $\mathrm{d}a_1, \dotsc, \mathrm{d}a_{\ell}$. Rather, $\mathrm{d}a_1, \dotsc, \mathrm{d}a_{\ell}$ are linearly independent if and only if $a_1, \dotsc, a_{\ell}$  are separably algebraically independent. Furthermore, pullback maps on differentials are no longer automatically injective; they are injective if and only if the corresponding field extension is separable. This fails for the variety of symmetric matrices of rank at most $d$ in characteristic $2$: the map
$$\mathbf{k}^{nd} \to \{n \times n \text{ symmetric matrices of rank }\le d\}, \quad A \mapsto AA^t$$
does not induce a separable extension of function fields. In all other cases, the analogous map does induce a separable extension of function fields. This can be proved by verifying that the matrix defining $\mathcal{S}_n(d)$ has rank $nd - \binom{d}{2}$ in characteristic $p \not= 2$, the matrix defining $\mathcal{H}_n(2d)$ has rank $2dn - \binom{2d + 1}{2}$ in any characteristic, and the matrix defining $\mathcal{B}_{m,n}(d,d)$ has rank $d(m + n) - d^2$ in any characteristic. 
Then the proof of Theorem~\ref{thm:equivalence} gives the following result.

\begin{theorem}
Let $\mathbf{k}$ be a field of characteristic $p$. 
\begin{enumerate}
\item If $p \not= 2$, then a collection of elements is independent in the dual of the symmetric power matroid $\mathrm{S}_n(n-d, p)$ if and only if the corresponding rational functions on the variety of $n \times n$ symmetric matrices of rank at most $d$ over $\mathbf{k}$ are separably algebraically independent.
\item A collection of elements is independent in the dual of the wedge power matroid $\mathrm{W}_n(n-2d, p)$ if and only if the corresponding rational functions on the variety of $n \times n$ skew-symmetric matrices\footnote{In characteristic $2$, we say that a matrix is skew-symmetric if it is symmetric with zeroes on the diagonal. It remains true that, over an algebraically closed field, an $n \times n$ matrix can be written as $AB^t - BA^t$ with $A, B$ $n \times d$ matrices if and only if it is skew-symmetric and has rank at most $2d$.} of rank at most $2d$ over $\mathbf{k}$ are separably algebraically independent.
\item A collection of elements is independent in the dual of the tensor matroid $\mathrm{T}_{m, n}(m -d , n- d, p)$ if and only if the corresponding rational functions on the variety of $m \times n$ matrices of rank at most $d$ over $\mathbf{k}$ are separably algebraically independent.
\end{enumerate}
\end{theorem}

\begin{remark}
A collection of coordinate functions $\{x_{ij} : (i, j) \in S \subseteq [m] \times [n]\} \subseteq K(Y_d)$ is algebraically independent if and only if the coordinate projection $Y_d \hookrightarrow \mathbf{k}^{mn} \to \mathbf{k}^S$ is \emph{dominant}, i.e., its image contains a Zariski open set. The collection is separably algebraically independent if and only if the coordinate projection is \emph{generically smooth}. 
\end{remark}

\begin{remark}
One can sometimes show that the rank completion matroid in characteristic $p$, i.e., the algebraic matroid of the rank at most $d$ locus, is independent of the characteristic by showing that the tropicalization of the rank at most $d$ locus does not depend on the characteristic. For example, this was shown for $d=2$ in \cite{tropicalrank}*{Section 6}. We note that results of this form do \emph{not} imply that the tensor matroid is independent of the characteristic. 
\end{remark}

\begin{remark}
We are unaware of any case in which the symmetric matrix completion matroid, the skew-symmetric matrix completion matroid, or the matrix completion matroid depends on the characteristic. In particular, the $4 \times 4$ symmetric rank $2$ matrix completion matroid is the same in all characteristics, even though $\mathrm{S}_4(2, 0) \not= \mathrm{S}_4(2, 2)$. 
\end{remark}

\section{Bipartite rigidity when $m-a$ is small}\label{sec:m-asmall}
In this section, we give a complete description of the matroid $\mathrm{T}_{m,n}(s,r, p)$ when $s \le 3$. Our description is independent of the characteristic $p$. 
By Theorem~\ref{thm:equivalence}(\ref{bipartitecase}), this gives a description of $\mathcal{B}_{m,n}(a,b)$ when $m - a \le 3$. We use this description to prove Theorem~\ref{thm:smallm-a} and Corollary~\ref{cor:laman}.

\subsection{The disjoint case and $s=1$}\label{subsec:descriptiontensor}

Given an infinite field $\F$ of characteristic $p$, consider $m$ generic vectors $u_1, \dotsc, u_m$ in $\F^s$ and $n$ generic vectors $v_1,\dotsc,v_n$ in $\F^r$. By definition, the independent sets of the matroid $\mathrm{T}_{m,n}(s,r,p)$ are given by sets $E\subseteq [m]\times [n]$ such that $\{u_i\otimes v_j : (i,j)\in E\}$ is linearly independent. We give a characteristic-independent description of the independent sets when $s\le 3$.

We write $E=\bigcup_{i=1}^m\{i\}\times A_i,A_i\subseteq [n]$. We first show that when the $A_i$ are disjoint a simple dimension/size criterion is equivalent to independence.

\begin{lemma}[The disjoint case]\label{lem:disjoint-case}
   Let $A_1, \hdots A_m \subseteq [n]$ be disjoint sets. Let $E=\bigcup_{i=1}^m \{i\}\times A_i\subseteq [m]\times [n]$. Then $E$ is independent in $\mathrm{T}_{m,n}(s,r, p)$ if and only if $|A_i| \le r$ for all $i$ and $\sum_{i=1}^m |A_i| \le sr$.   
 \end{lemma}
 \begin{proof}
   The ``only if'' conditions follow from comparing dimensions. For the ``if'' direction, it suffices to give a choice of the $u_i$ and $v_j$ such that $\{u_i \otimes v_j : (i, j) \in E\}$ is a set of linearly independent vectors. Let $e_1, \hdots, e_{r}$ be a basis of $\F^{r}$. Set each $v_j \in A_1$ to be one of the basis vectors. Keep going, and set each $v_j \in A_2$ to be the ``next'' basis vector and so forth. Each $v_j$ will be set to some basis vector, no two vectors in some $A_i$ are set to the same basis vector, and each basis vector is used at most $s$ times. For each $j \in [r]$, let $C_j$ be the set of $i$ for which $e_j$ is used for some vector in $A_i$. 

   Pick the $u_i$ to be generic vectors, so every $s$ of them are linearly independent. Consider a linear dependence among $\{u_i \otimes v_j : (i, j) \in E\}$. That is, some $\lambda_{i,j} \in \F$ for which
   \[
     \sum_{(i,j) \in E} \lambda_{i,j} u_i\otimes v_j = 0.
   \]
   In particular, we have that
   \[
     \sum_{j \in [r]} \sum_{i \in C_j} \lambda_{i,j'} u_i \otimes e_j = 0,
   \]
   where $j'$ is shorthand for the $j' \in A_i$ for which $v_{j'} = e_j$.  Since the $e_j$ form a basis, we have that, for all $j \in [r]$,
   \[
     \sum_{i \in C_j} \lambda_{i,j'} u_i = 0. 
   \]
   Since $|C_j| \le s$ for each $j$ and the $u_i$ are generic, we have that $\lambda_{i,j'} = 0$ for all $(i, j') \in E$. Thus, the vectors $u_i \otimes v_j$ are linearly independent, as desired.
 \end{proof}

We can use the above result to address the case $s = 1$.

\begin{corollary}\label{lem:r=1}
Let $A_1, \hdots A_m \subseteq [n]$, and set $E=\bigcup_{i=1}^m \{i\}\times A_i\subseteq [m]\times [n]$.
Then $E$ is independent in $\mathrm{T}_{m,n}(1, r, p)$ if and only if
\begin{align}
  |A_i \cap A_j| &= 0 && \text{ for all distinct } i,j\in [m]\label{eq:A01}\\
  \sum_{i=1}^m |A_i| &\le r.\label{eq:A03}
\end{align}
\end{corollary}
\begin{proof}
First, we prove that (\ref{eq:A01}) and (\ref{eq:A03}) are necessary. If there exists some $k \in A_i \cap A_j$, then $u_i \otimes v_k$ and $u_j \otimes v_k$ are linearly dependent, so (\ref{eq:A01}) is necessary. The necessity of (\ref{eq:A03}) is obvious, as $\dim (\F \otimes \F^r) = r$.

To prove sufficiency, note that (\ref{eq:A01}) implies we are in the disjoint case. Thus, we can apply Lemma~\ref{lem:disjoint-case} with $s=1$.
\end{proof}

\subsection{The case $s=2$}

We now give a description of $\mathrm{T}_{m,n}(2,r,p)$ for every $p$ and $r$. A different description of $\mathcal{B}_{m,n}(m -2, m-2)$ was given in \cite[Proposition 2]{Tsakiris}, which gives a description of $\mathrm{T}_{m,n}(2,r,0)$. 

\begin{proposition}[Case $s=2$]\label{lem:m-a=2}
Let $A_1, \hdots A_m \subseteq [n]$ be sets of size at most $r$, and set $E=\bigcup_{i=1}^m \{i\}\times A_i\subseteq [m]\times [n]$.
Then $E$ is independent in $\mathrm{T}_{m,n}(2, r, p)$ if and only if
\begin{align}
  |A_i \cap A_j \cap A_k| &= 0 && \text{ for all distinct } i,j,k\in [m]\label{eq:A1}\\
  \sum_{1 \le i < j \le m} |A_i \cap A_j| + \left|A_k \setminus \bigcup_{i\in [m]\setminus \{k\}}  A_i\right| &\le r && \text{ for all } k \in [m]\label{eq:A2}\\
  \sum_{i=1}^m |A_i| &\le 2r.\label{eq:A3}
\end{align}
\end{proposition}

\begin{proof}
  For any set $A\subseteq [n] $, let $V_A$ be the subspace of $\F^r$ spanned by the generic vectors $v_i$ for $i\in A$.

  The fact that these inequalities are necessary is not hard to see. The last inequality is equivalent to requiring $|E| \le 2r=\dim (\F^2\otimes \F^r)$.  For (\ref{eq:A2}), let $W = \Span \{u_i \otimes v_j : (i,j) \in E\}$. Observe that, for all $i < j \in [m]$, $\F^2 \otimes V_{A_i \cap A_j} \subseteq W$. In particular, $u_k \otimes V_{A_i \cap A_j} \subseteq W$. Also note that $u_k \otimes V_{A_k \setminus \bigcup_{i\not = k} A_i} \subseteq W$. For these to be linearly independent we must have (\ref{eq:A2}).

  If $\ell\in A_i\cap A_j\cap A_k$ for some distinct $i,j,k \in [m]$, then we have linearly dependent vectors $u_i\otimes v_\ell,u_j\otimes v_\ell,u_k\otimes v_\ell\in \F^2\otimes \F^r$. This shows that \eqref{eq:A1} is necessary.
  
  If the above inequalities are satisfied, then we show that $\{u_i \otimes v_j : (i,j) \in E\}$ are independent for a semi-explicit choice of the $u_i \in \F^2$ and $v_j\in \F^{r}$. Let $e_1, \hdots, e_{r}$ be a basis of $\F^{r}$. Let $S_1$ be the elements of $\bigcup_{i=1}^m A_i$ that appear in exactly one $A_i$. Let $S_2$ be the elements which appear in exactly two $A_i$. By (\ref{eq:A1}), we have that $S_1 \cup S_2 = \bigcup_{i=1}^m A_i$. Observe then that (\ref{eq:A2}) and (\ref{eq:A3}) can thus be re-written as
  \begin{align}
    |A_i \cap S_1| + |S_2| &\le r \label{eq:B2} \qquad \qquad \text{ for all }i\\
    |S_1| + 2|S_2| &\le 2r \label{eq:B3} 
      \end{align}

  Pick $u_1, \hdots, u_m$ to be generic vectors in $\F^2$, so any pair of the vectors are linearly independent. For each $j \in S_2$, set $v_j$ to be a separate $e_i$ for $i \in \{1, 2, \hdots, |S_2|\}.$ This is possible by (\ref{eq:A2}) as $|S_2| \le r$. Finally, set the $v_i$ for $i \in S_1$ to be generic vectors in the space spanned by $\{e_i : i \in \{|S_2|+1, \hdots, r\}\}$.

  It is not hard to see from these choices that $\{u_i \otimes v_j : (i, j) \in E\}$ are independent if and only if $\{u_i \otimes v_j : (i, j) \in E, j \in S_1\}$ are independent, as the vectors in $S_2$ and $S_1$ are in disjoint subspaces. Thus a non-zero linear combination of $\{u_i \otimes v_j : (i, j) \in E, j \in S_2\}$ cannot intersect with the subspace spanned by $\{u_i \otimes v_j : (i, j) \in E, j \in S_1\}$. As $j\in S_2$ only appears in two of the $A_i$, we also see that the set $\{u_i \otimes v_j : (i, j) \in E, j \in S_2\}$ is linearly independent and indeed spans $\F^2\otimes \F^{|S_2|}$.
  
Then $\{u_i \otimes v_j : (i, j) \in E, j \in S_1\}$ are vectors in the subspace $\F^2\otimes \F^{r-|S_2|}$. The result follows from applying Lemma~\ref{lem:disjoint-case} with the parameters $(m,n,s,r)=(m, n-|S_2|, 2, r - |S_2|)$ and the sets $A_i \cap S_1$ for $i \in [m]$. The inequalities needed to invoke Lemma~\ref{lem:disjoint-case} are exactly (\ref{eq:B2}) and (\ref{eq:B3}). 
\end{proof}

\subsection{The case $s=3$}

\begin{proposition}[Case $s=3$]\label{lem:m-a=3}
Let $A_1, \hdots A_m \subseteq [n]$ be sets of size at most $r$, and set $E=\bigcup_{i=1}^m \{i\}\times A_i\subseteq [m]\times [n]$. For $k \in \{1, 2, 3\}$, let $S_k$ be the set of $j \in [n]$ which appear in exactly $k$ of the $A_i$. 
Then $E$ is independent in $\mathrm{T}_{m,n}(3, r, p)$ if and only if
  \begin{align}
    |A_i \cap A_j \cap A_k \cap A_\ell| &= 0 && \text{for all distinct }i,j,k,\ell\in [m]\label{eq:c1}\\
    |A_i \setminus S_3| + |S_3| &\le r && \text{for all }i\in [m]\label{eq:c2}\\
    |(A_i \cap A_j) \setminus S_3| + |(A_k \cap A_\ell) \setminus S_3| + |S_3| &\le r  && \text{for all distinct }i,j,k,\ell\in [m] \label{eq:c3}\\
    |A_i \setminus S_3| + |A_j \setminus S_3| + |S_2 \setminus (S_3 \cup A_i \cup A_j)| + 2|S_3| &\le 2r && \text{for all distinct }i,j\in [m] \label{eq:c5}\\
    |S_1| + 2|S_2| + 3|S_3| &\le 3r \label{eq:c4}.
  \end{align}
\end{proposition}

\begin{figure}[t]
\begin{center}
\begin{tikzpicture}
\newcommand{\cbb}{\Huge $\star$}
\newcommand{\cbr}{\Huge $\diamond$}

\fill[red!50!white] (0,4) rectangle (3,3);
\fill[red!50!white] (3,3) rectangle (6,2);

\fill[red!50!white] (7,2) rectangle (9,1);
\fill[red!50!white] (8,2) rectangle (9,0);
\fill[red!50!white] (6,1) rectangle (7,-1);
\fill[red!50!white] (6,0) rectangle (8,-1);

\fill[blue!40!white] (0,-1) rectangle (6,2);

\draw (0.5,3.5) node {\cbr};
\draw (1.5,3.5) node {\cbr};
\draw (2.5,3.5) node {\cbr};

\draw (3.5,2.5) node {\cbr};
\draw (4.5,2.5) node {\cbr};
\draw (5.5,2.5) node {\cbr};

\draw (6.5,0.5) node {\cbr};
\draw (7.5,1.5) node {\cbr};
\draw (8.5,1.5) node {\cbr};
\draw (6.5,-0.5) node {\cbr};
\draw (7.5,-0.5) node {\cbr};
\draw (8.5,0.5) node {\cbr};
\draw (0.5,1.5) node {\cbb};
\draw (0.5,0.5) node {\cbb};
\draw (1.5,1.5) node {\cbb};
\draw (1.5,0.5) node {\cbb};
\draw (2.5,1.5) node {\cbb};
\draw (2.5,0.5) node {\cbb};

\draw (0.5,-0.5) node {\cbb};
\draw (1.5,-0.5) node {\cbb};
\draw (2.5,-0.5) node {\cbb};
\draw (3.5,-0.5) node {\cbb};
\draw (4.5,-0.5) node {\cbb};
\draw (5.5,-0.5) node {\cbb};

\draw (3.5,1.5) node {\cbb};
\draw (3.5,0.5) node {\cbb};
\draw (4.5,1.5) node {\cbb};
\draw (4.5,0.5) node {\cbb};
\draw (5.5,1.5) node {\cbb};
\draw (5.5,0.5) node {\cbb};

\draw (0,-1) grid (9,4);
\end{tikzpicture}
\end{center}
\caption{(red $\diamond$) A circuit of $\mathrm{T}_{5,9}(3,4,p)$ which violates \eqref{eq:c5} but none of \eqref{eq:c1},\eqref{eq:c2},\eqref{eq:c3}, or \eqref{eq:c4}. (blue $\star$) A corresponding Laman circuit in $\mathcal{B}_{5,9}(2,5)$. See Figure~\ref{fig:nonlaman} for how to interpret.}\label{fig:5925}
\end{figure}

\begin{example}
A natural question that one may ask concerning Lemma~\ref{lem:m-a=3} is whether each of these constraints is ``strictly'' necessary. We note that (\ref{eq:c1}), (\ref{eq:c2}), and (\ref{eq:c4}) capture many simple Laman conditions, showing that sets such as $E = [4] \times [1]$, $E = [1] \times [r+1]$, and $E = \{(i,i) : i \in [3r+1]\}$ are dependent in $\mathrm{T}_{m,n}(3,r,p)$.

The role of (\ref{eq:c5}) is to capture more complex Laman conditions. To see why, assuming for simplicity that $|E| \le 3r$ and that $S_3 =  \emptyset$. Pick the sets $S = [m] \setminus \{i,j\}$ and $T = [n] \setminus (S_2 \setminus (A_i \cup A_j))$. Note that $|S| = m-2 = a+1$ and that $E \cap (S \times T) = \emptyset$. Therefore, the Laman condition imposes that $|T| \le n-r$. Thus, $|S_2 \setminus (A_i \cup A_j)| \ge r$. Since $3r \ge |E| \ge |A_1| + |A_2| + 2|S_2 \setminus (A_i \cup A_j)|$, we may deduce (\ref{eq:c5}). See Figure~\ref{fig:5925} for an example a circuit caught by (\ref{eq:c5}) and the corresponding Laman condition violation.

The role of (\ref{eq:c3}) is to systematically capture non-Laman circuits similar to the one described in Figure~\ref{fig:nonlaman}. 
\end{example}

\begin{proof}
Let $W = \Span \{u_i \otimes v_j : (i,j) \in E\}$.   For any set $A\subseteq [n] $, let $V_A$ be the subspace of $\F^r$ spanned by the generic vectors $v_i$ for $i\in A$.

The necessity of (\ref{eq:c1}) and (\ref{eq:c4}) is obvious. To show the necessity of \eqref{eq:c2}, observe that $\F^3 \otimes V_{S_3} \subseteq W$. Thus, $u_i \otimes V_{S_3} \subseteq W$. Also note that $u_i \otimes V_{A_i \setminus S_3} \subseteq W$. For these two subspaces to be linearly independent we must have $|A_i \setminus S_3|+|S_3|\le r$.

Next, we show \eqref{eq:c3} is necessary. Consider distinct $i,j,k,\ell \in [m]$ and let $u$ be a nontrivial vector in $\Span\{u_i, u_j\} \cap \Span\{u_k, u_\ell\}$. As before, $u \otimes V_{S_3} \subseteq W$. Further, $u \otimes V_{(A_i \cap A_j) \setminus S_3} \subseteq W$ and $u \otimes V_{(A_k \cap A_\ell) \setminus S_3} \subseteq W$. If these subspaces are linearly independent, then (\ref{eq:c3}) must hold. 

To finish the proof of the necessity of the inequalities, we show \eqref{eq:c5} is necessary. Consider distinct $i,j \in [m]$. For each $k \in S_2 \setminus (S_3 \cup A_i \cup A_j)$, let $u'_k$ be a nontrivial vector in $\Span \{u_i, u_j\} \cap \Span \{u_a, u_b\}$, where $a$ and $b$ are the two sets such that $k \in A_a \cap A_b$. Observe that $u_i \otimes V_{A_i \setminus S_3}, u_j \otimes V_{A_j \setminus S_3}, \sum_{k \in S_2 \setminus (S_3 \cup A_i \cup A_j)} u'_k \otimes v_{k},$ and $\Span\{u_i, u_j\} \otimes V_{S_3}$ are linearly independent subspaces of $W \cap (\Span\{u_i, u_j\} \otimes \F^{r})$.  This proves the necessity of \eqref{eq:c5}.

Next, we prove the sufficiency of these equations by specializing some of the vectors. The strategy is to induct by  finding a nice subspace to quotient by. We first use this strategy to reduce to the case $S_3=\emptyset$.

\begin{claim}
$E$ is independent in $\mathrm{T}_{m,n}(3,r,p)$ if and only if $E \setminus \{(i, j) : j \in S_3\}$ is independent in $\mathrm{T}_{m,n}(3,r-|S_3|,p)$.
\end{claim}
\begin{proof}
Let $e_1,\hdots,e_r$ be a basis of $\F^r$. For each $i\in S_3$, we set $v_i$ to a different basis element (which is possible as $|S_3|\le r$ by \eqref{eq:c4}). For $j\in [n]\setminus S_3$ we choose the $v_j$ to be generic vectors in the subspace spanned by $e_{|S_3|+1},\hdots,e_r$. We have that $E \setminus \{(i, j) : j \in S_3\}$ is independent in $\mathrm{T}_{m,n}(3,r-|S_3|,p)$ if and only if $\{u_i\otimes v_j : (i,j)\in E,j\not\in S_3\}$ is independent in $\mathrm{T}_{m,n}(3,r,p)$.
We see that
$$\operatorname{span} \{u_i\otimes v_j : (i,j)\in E,j\in S_3\} \cap \operatorname{span} \{u_i\otimes v_j : (i,j)\in E,j\not\in S_3\} = 0.$$ As the $u_i$ are generic, $\{u_i\otimes v_j : (i,j)\in E,j\in S_3\}$ is independent in $\mathrm{T}_{m,n}(3,r,p)$ if and only if $E \setminus \{(i, j) : j \in S_3\}$ is independent in $\mathrm{T}_{m,n}(3,r-|S_3|,p)$, so this implies the result. 
\end{proof}

 With the assumption $S_3 = \emptyset$, we have a much simpler set of inequalities:
\begin{align}
|A_i| &\le r && \text{ for all } i\in [n]\label{eq:d1}\\
|A_i \cap A_j| + |A_k \cap A_\ell| &\le r && \text{ for all distinct } i,j,k, \ell \in [n]\label{eq:d2}\\
|A_i| + |A_j| + |S_2 \setminus (A_i \cup A_j)| &\le 2r && \text{ for all distinct } i,j\in [n]\label{eq:d3}\\
|S_1| + 2|S_2| &\le 3r.\label{eq:d4}
\end{align}

We will proceed via induction on $m$ and considering several cases.

\paragraph{\bf Case 1, \eqref{eq:d1} is tight:} Without loss of generality, assume that $i=1$ in \eqref{eq:d1}. Since $|A_1|=r$, we have that $V_{A_1} = \textbf{k}^r$. 
For $i \ge 2$, let $u_i'$ be the image of $u_i$ in $\F^3/\langle u_1 \rangle$. Then the $u'_i$ are generic vectors in $\F^3/\langle u_1 \rangle$.

We claim that the $(i, j) \in E$ with $i \ge 2$ give an independent set in $\mathrm{T}_{m-1,n}(2,r,p)$. This follows by checking the inequalities of Lemma~\ref{lem:m-a=2}: \eqref{eq:A1} follows because $|S_3|=0$, \eqref{eq:A2} follows from \eqref{eq:d3} for $A_1$ and $A_i$, and \eqref{eq:A3} follows from \eqref{eq:d4}. This means $\{u'_i\otimes v_j : (i,j)\in \bigcup_{k=2}^m \{k\}\times A_k\}$ is  linearly independent, which implies $\{u_i\otimes v_j :  (i,j)\in E\}$ is linearly independent.

\paragraph{\bf Case 2, \eqref{eq:d2} is tight:} Without loss of generality, assume that $(i,j,k,\ell) = (1,2,3,4)$. As the $u_i$ are generic vectors, any three of them are linearly independent and there is no common non-zero vector in the span of any three disjoint pairs $(u_{i_1},u_{i_2}),(u_{i_3},u_{i_4}),(u_{i_5},u_{i_6})$.

Suppose $|A_1\cap A_2|+|A_3\cap A_4|=r$, so $\textbf{k}^r = V_{A_1 \cap A_2} + V_{A_3 \cap A_4}$. Let $u$ be a non-zero vector in $\Span\{u_1, u_2\} \cap \Span\{u_3, u_4\}$.  We quotient $\textbf{k}^3 \otimes \textbf{k}^r$ by $  u\otimes V_{A_1\cap A_2} +    u\otimes V_{A_3\cap A_4} =   u \otimes \textbf{k}^r$. Let $u'_j$ be the image of $u_j$ in $\F^3/\langle u \rangle$.

We set $A'_1=A_1 \cup A_2$ and $A'_2=A_3 \cup A_4$. We note $u'_1,u'_2$ are scalar multiples of each other and so are $u'_3,u'_4$. Observe that  $\{u'_1,u'_3,u'_5, \dotsc, u'_m\}$ is a set of generic vectors, i.e., any pair is linearly independent. 

Let $E'=(\{1\}\times A'_1) \cup (\{3\}\times A'_2) \cup \bigcup_{i=5}^m (\{i\}\times A_i)$. We claim that $\{u'_i\otimes v_j : (i,j)\in E'\}$ is linearly independent in $(\F^3/\langle u \rangle)\otimes \F^r$. This follows by checking the conditions in Lemma~\ref{lem:m-a=2} for $A'_1,A'_2,A_5,\hdots,A_m$. Indeed, \eqref{eq:A1} easily follows. Let $B = (A_1 \cap A_2) \cup (A_3 \cap A_4)$, which has size exactly $r$ because $A_1 \cap A_2 \cap A_3 \cap A_4 = \emptyset$. We see that $|A'_1|+|A'_2|+\sum_{i=5}^m |A_i|=\sum_{i=1}^m |A_i|-|B|\le 2r$, so \eqref{eq:A3} holds. To check \eqref{eq:A2}, we use \eqref{eq:d3} for $A_1,\hdots,A_m$ and the fact that the indices which appear in two sets for $A'_1,A'_2,A_5,\hdots,A_m$ are precisely $S_2\setminus B$, and $|S_2\setminus B|=|S_2|-|B|=|S_2|-r$.  

If we have a non-zero linear combination of $\{u_i\otimes v_j : (i,j)\in E\}$ which is zero, then quotienting by $u\otimes \F^r$ and using the linear independence of $\{u'_i\otimes v_j : (i,j)\in E'\}$ gives a contradiction.

\paragraph{\bf Case 3, \eqref{eq:d3} is tight:} Without loss of generality, assume that $|A_1|+|A_2|+|S_2\setminus (A_{1}\cup A_2)|=2r$.  For each $i \in S_2 \setminus (A_1 \cup A_2)$, let $u'_i$ be a non-zero element of $\Span\{u_1,u_2\} \cap \Span\{u_j,u_k\}$, where $j, k$ are the two indices such that $i \in A_j \cap A_k$. 

\begin{claim}\label{claim:m-a=3C3Span}
The three subspaces $\langle u_1\rangle \otimes V_{A_1}, \langle u_2\rangle \otimes V_{A_2},$ and $\sum_{i \in S_2 \setminus (A_1 \cup A_2)} \langle u'_i\rangle  \otimes \langle v_{i}\rangle $ span $\Span\{u_1, u_2\} \otimes \F^{r}.$
\end{claim}
\begin{proof}
We see that $\langle u_1\rangle  \otimes V_{A_1}, \langle u_2\rangle \otimes V_{A_2},$ and  $\sum_{i \in S_2 \setminus (A_1 \cup A_2)} \langle u'_i \rangle \otimes \langle v_{i} \rangle $ span a subspace of $\Span\{u_1, u_2\} \otimes \F^{r}$.
We will show it spans the whole space by proving that $\{u_1 \otimes v_i:i\in A_1\}, \{u_2 \otimes v_j: j\in A_2\},$ and  $\{u'_i \otimes v_{i} : i \in S_2 \setminus (A_1 \cup A_2)\}$  are linearly independent. 

We will use Lemma~\ref{lem:m-a=2} for this. First note that, because the $u_i$ are generic, the vectors $\{u_1, u_2\} \cup \{u'_i : i\in S_2 \setminus (A_1 \cup A_2)\}$ are generic vectors in $\Span\{u_1, u_2\}$. In the collection of sets $\{A_1,A_2\} \cup \{\{i\} : i \in S_2 \setminus (A_1 \cup A_2)\}$, the only non-empty pairwise intersection is between $A_1$ and $A_2$. As $|A_1\cap A_2|<r$, we can use this to check \eqref{eq:A1},\eqref{eq:A2}, and \eqref{eq:A3}.
\end{proof}

Suppose there is a dependence relation in  $\{u_i\otimes v_j :  (i,j)\in E\}$:
\begin{equation}\label{eq:m-a=3C3}
    \sum\limits_{(i,j)\in E} \lambda_{i,j} u_i\otimes v_j = 0.
\end{equation}

Consider the image of this relation in $\mathbf{k}^3/ \Span\{u_1, u_2\} \otimes \mathbf{k}^r$. All $u_i,i\ge 3$ are projected to non-zero scalar multiples, say $\alpha_i$, of the same vector $u$. Then \eqref{eq:m-a=3C3} becomes,
\begin{equation}\label{eq:m-a=3C3proj}
    \sum\limits_{i=3}^{m}\sum\limits_{j\in A_i\setminus S_2} \alpha_i \lambda_{i,j} u\otimes v_j + \sum\limits_{i\ne j\in [m]\setminus \{1,2\}}\sum\limits_{k\in A_i\cap A_j} (\alpha_i \lambda_{i,k}+\alpha_j \lambda_{j,k}) u\otimes v_k=0. 
\end{equation}

Note $\{v_j : j\in A_i\setminus S_2,i\ge 3\} \cup \{v_k : k\in A_i\cap A_j,i\ne j\in [m]\setminus \{1,2\}\}$ is a set of at most $r$ distinct vectors (we use $S_1+2S_2\le 3r$ and $|A_1|+|A_2|+|S_2\setminus (A_1\cup A_2)|=2r$). Thus $\lambda_{i,j}=0$ for $i\ge 3, j\in A_i\setminus S_2$ and $\alpha_i \lambda_{i,k}+\alpha_j \lambda_{j,k}=0$ for $i\ne j\in [m]\setminus \{1,2\},k\in A_i\cap A_j$.

Using this in \eqref{eq:m-a=3C3} gives us a linear dependence  which contradicts Claim~\ref{claim:m-a=3C3Span}.

\paragraph{\bf Case 4, $|S_2|=0$:}
This is just the disjoint case (Lemma~\ref{lem:disjoint-case}).

\paragraph{\bf Case 5, $|S_1|\ne 0,|S_2|\ne 0$ and \eqref{eq:d1},\eqref{eq:d2},\eqref{eq:d3} are not tight:}
As $m \ge 3$, without loss of generality we can assume that $A_1 \cap S_1 \neq \emptyset$ and $A_2 \cap A_3 \neq \emptyset$. After relabeling, let $i \in A_1$ be in $S_1$ and $j\in A_2\cap A_3$. We now set $v_i=v_j=e_1$ and quotient by $\F^3\otimes e_1\cong \Span\{u_1,u_2,u_3\}\otimes e_1$. Set $A'_1=A_1\setminus \{i\}$, $A'_2=A_2\setminus \{j\}$, $A'_3=A_3\setminus \{j\}$, and $A'_k=A_k$ for $k\ge 4$. Then $A'_1,\hdots,A'_m$ satisfy the inequalities \eqref{eq:d1},\eqref{eq:d2},\eqref{eq:d4} for $\mathbf{k}^3 \otimes 
\mathbf{k}^{r-1}$. The only nontrivial check is for \eqref{eq:d3}, but if we re-write \eqref{eq:d3} as $|S_2\cup A_1\cup A_2| + |A_1\cap A_2|\le 2r-1$ (as \eqref{eq:d3} is not tight) then we see $A'_i$ satisfy the needed inequality.

\paragraph{\bf Case 6, $|S_1|= 0,|S_2|\ne 0$ and \eqref{eq:d1},\eqref{eq:d2},\eqref{eq:d3} are not tight:}
If \eqref{eq:d4} is not tight then $2|S_2|\le 3r-1$. Choose $i \in S_2$ and quotient by $\mathbf{k}^3 \otimes \langle v_i \rangle$. The new sets will satisfy the inequalities for $\mathbf{k}^3 \otimes \mathbf{k}^{r-1}$.

From now on we assume \eqref{eq:d4} is tight, which means $|S_2| = 3r/2$ (which also implies that $r$ is even). In that case (\ref{eq:d3}) simplifies even further to $|A_i \cap A_j| \le r/2-1$ (as \eqref{eq:d3} is not tight), which also makes (\ref{eq:d2}) redundant.

We claim that if $r \ge 1$, we can always pick $i_1, i_2, i_3 \in S_2$  with the following properties. Let $j_1,j'_1,j_2,j'_2,j_3,j'_3 \in [m]$ be such that $i_1 \in A_{j_1} \cap A_{j'_1}$, and so forth. Further, set $A'_i = A_i \setminus \{i_1,i_2,i_3\}$ for all $i \in [m]$. It is not hard to see that $|A'_1| + \cdots + |A'_m| = 3(r-2)$. We claim that the following holds:
\begin{itemize}
\item If $e_1, e_2 \in \F^2$ are standard basis vectors, then $\Span\{u_{j_1}, u_{j'_1}\} \otimes e_1 + \Span\{u_{j_2}, u_{j'_2}\} \otimes e_2 + \Span\{u_{j_3}, u_{j'_3}\} \otimes (e_1 + e_2) = \F^3 \otimes \F^2$.
\item $|A'_i| \le r-2$ for all $i \in [m]$.
\item $|A'_i \cap A'_j| \le (r-2)/2$ for $i \neq j \in [m]$.
\end{itemize}

If all these properties are satisfied, then we can use induction to finish this case. By setting $v_{i_1}=e_1,v_{i_2}=e_2,$ and $v_{i_3}=e_1+e_2$ and quotienting out $\F^3 \otimes \Span\{e_1,e_2\}$, we reduce to a smaller case on $\F^3 \otimes \F^{r-2}$. 

The third condition follows directly as $|A_i\cap A_j|\le r/2-1$.
One can check that the first condition is satisfied when the pairs $\{j_1,j'_1\},\{j_2,j'_2\},\{j_3,j'_3\}$ are distinct and no index appears in the multiset $\{j_1,j'_1,j_2,j'_2,j_3,j'_3\}$ more than twice.\footnote{And this is not a characteristic-dependent condition.} As $|A_i\cap A_j|\le r/2-1$ for every $i,j$ and $S_2=3r/2$,  we have at least three pairs to choose from. We would be forced to pick the same index thrice if there are only 4 sets $A_1,A_2,A_3,A_4$ and one of them intersects with the rest but the rest do not intersect with each other, but that would contradict that $S_2=3r/2$.

For the second condition, if $|A_i| > r-2$, then $i$ must appear at least $|A_i| - (r-2)$ times in $\{j_1,j'_1,j_2,j'_2,j_3,j'_3\}$. As $|A_i|\le r-1$ we only have to ensure $|A_i|=r-1$ happens at most 6 times. If it happens 7 times, say for $A_1,\hdots,A_7$ then $|A_i\cap A_j|\ge r/2-2$ for $i\ne j=1,\hdots,7$. This gives us that $21r/2-42\le |S_2|=3r/2$ which implies $r\le 42/9$. As $r$ is even we have $r\in \{2,4\}$. In either case we have $|S_2|\le 6$ and $|A_i|=r-1$ for $i=1,\hdots,7$, so the fact that $|S_3|=|S_1|=0$ leads to a contradiction.
\end{proof}

\begin{proof}[Proof of Theorem~\ref{thm:smallm-a}]
The description of $\mathrm{T}_{m,n}(s,r,0)$ when $s \le 3$ above shows that we can check independence in $\mathrm{T}_{m,n}(s,r,0)$ by checking polynomially many conditions (the case $s=0$ is immediate).  By Theorem~\ref{thm:equivalence}, $\mathrm{T}_{m,n}(s,r,0)$ is dual to $\mathcal{B}_{m,n}(m-s,n-r)$, so a set is independent in $\mathrm{T}_{m,n}(s,r,0)$ if and only if its complement is spanning in $\mathcal{B}_{m,n}(m-s,n-r)$. 
By \cite{MatroidComplexity}, there is a polynomial time algorithm to compute the rank function of a matroid by checking if polynomially many sets are spanning. 
\end{proof}

\subsection{Characterizing the Laman condition}

In this section, we prove Corollary~\ref{cor:laman}. One consequence of Corollary~\ref{cor:laman} is that, when $m - a \le 2$, the Laman circuits are the circuits of a matroid; Example~\ref{ex:nonlaman} shows that this is false in general. It will be necessary to prove this directly in order to prove Corollary~\ref{cor:laman}. 

Any subset $U$ of $[m] \times [n]$ is contained in some \emph{minimal rectangle} $S \times T$. We say that $U$ \emph{violates the Laman condition} if there is some rectangle $S \times T$ with $|S| \ge a, |T| \ge b$ such that $|U \cap S \times T| > |S|b + |T|a - ab$. Clearly any subset which violates the Laman condition contains a minimal subset which violates the Laman condition. In general, a minimal subset which violates the Laman condition need not contain a Laman circuit (which are, by definition, minimal dependent sets).

\begin{lemma}\label{lem:circuitelim}
Let $C_1, C_2$ be distinct minimal sets which violate the Laman condition whose minimal rectangles are $S_1 \times T_1$ and $S_2 \times T_2$, respectively. Suppose that $|S_1 \cap S_2| \ge a$. Then, for any $e \in C_1 \cup C_2$, $C_1 \cup C_2 \setminus e$ violates the Laman condition. 
\end{lemma}

\begin{proof}
We have $|C_1| \ge b |S_1| + a |T_1| - ab + 1$ and $|C_2| \ge b|S_2| + a |T_2| - ab + 1$. Therefore
$$|C_1 \cup C_2 \setminus e| \ge b(|S_1| + |S_2|) + a(|T_1| + |T_2|) - 2ab + 1 - |C_1 \cap C_2|.$$
We claim that we have the bound 
$$b|S_1 \cap S_2| + a|T_1 \cap T_2| -ab \ge |C_1 \cap C_2|.$$
Given this inequality, we see that $|C_1 \cup C_2 \setminus e| \ge b|S_1 \cup S_2| + a |T_1 \cup T_2| -ab + 1$. As $C_1 \cup C_2 \setminus e$ is contained in the rectangle $(S_1 \cup S_2) \times (T_1 \cup T_2)$, this implies the result.

First suppose that $|T_1 \cap T_2| \ge b$. Note that, because $C_1 \not= C_2$, $C_1 \cap C_2$ does not violate the Laman condition and is contained in $(S_1 \cap S_2) \times (T_1 \cap T_2)$. The claimed inequality follows. 

Now suppose that $|T_1 \cap T_2| \le b$. Write $|S_1 \cap S_2| = a + \ell$ for some $\ell \ge 0$. Then we have that
$$b|S_1 \cap S_2| + a|T_1 \cap T_2| -ab = b\ell + a |T_1 \cap T_2|.$$
The bound $|C_1 \cap C_2| \le |S_1 \cap S_2| \cdot |T_1 \cap T_2| = (a + \ell) |T_1 \cap T_2|$ implies the claim, as $b\ell \ge \ell |T_1 \cap T_2|$. 
\end{proof}

\begin{lemma}
For $m, n, a, b$ with $m - a \le 2$, the minimal sets which violate the Laman condition in $[m] \times [n]$ form the circuits of a matroid. 
\end{lemma}

\begin{proof}
We must check the circuit elimination axiom. Let $C_1, C_2$ be distinct minimal sets which violate the Laman condition, and let $S_1 \times T_1$, $S_2 \times T_2$ be the minimal rectangles containing them. We have $|S_1| \ge a+1$ and $|S_2| \ge a+1$: if $|S_i| = a$, then no subset of $S_i \times T_i$ is large enough to violate the Laman condition. As $m - a \le 2$, this implies that $|S_1 \cap S_2| \ge a$. It then follows from Lemma~\ref{lem:circuitelim} that, for any $e \in C_1 \cup C_2$, $C_1 \cup C_2 \setminus e$ violates the Laman condition. 
\end{proof}

\begin{proposition}\label{prop:Lamanbipartite}
If $m - a \le 2$, then the circuits of $\mathcal{B}_{m,n}(a,b)$ are the minimal sets which violate the Laman condition.
\end{proposition}

\begin{proof}
We focus on the case $m - a = 2$; the cases $m  = a$ and $m = a + 1$ are straightforward or can be deduced from the case $m - a = 2$. Let $\mathrm{M}_{\Lam}$ be the matroid on $[m] \times [n]$ whose circuits are the minimal sets which violate the Laman condition. Note that every basis of $\mathcal{B}_{m,n}(a,b)$ is independent in $\mathrm{M}_{\Lam}$. The Laman condition implies that the rank of $\mathrm{M}_{\Lam}$ is at most $na + mb - ab = 2b + na$, which is the rank of $\mathcal{B}_{m,n}(a,b)$, so the rank of $\mathrm{M}_{\Lam}$ is the same as the rank of $\mathcal{B}_{m,n}(a,b)$. It therefore suffices to show that if $F \subseteq [m] \times [n]$ is independent in $\mathrm{M}_{\Lam}$ with $|F| = 2b + na$, then the complement $F^c$ is independent in $\mathrm{T}_{m,n}(2, n-b, 0)$. Set $F^c = \cup_i \{i\} \times A_i$. 

We check that $F^c$ satisfies the inequalities in Proposition~\ref{lem:m-a=2} (with $r = n-b$). That $|A_i| \le n - b$ follows from the Laman condition applied to $([m] \setminus \{i\}) \times [n]$ and the fact that $|F| = 2b + na$. 

Condition \eqref{eq:A3} holds: it is equivalent to requiring that $|F| \ge 2b + na$. 

Condition \eqref{eq:A1} holds: suppose $i \in A_j \cap A_k \cap A_\ell$, i.e., $|F \cap [m] \times \{i\}| < a$. The Laman condition implies that  $|F \cap [m] \times ([n] \setminus i)| \le 2b + a(n-1)$. But this implies that $|F| < 2b + an$, a contradiction. 

Condition \eqref{eq:A2} holds: let $S_1$ be the set of elements in $[n]$ which occur in exactly $1$ of the $A_i$, let $S_2$ be the elements in $[n]$ which occur in exactly two of the $A_i$, and suppose $|A_i \cap S_1| + |S_2| > n-b$. 
Consider the complement of the $i$th row. The Laman condition implies that there can be at most $b$ columns $j$ where $F$ contains $([m] \setminus i) \times \{j\}$. The other $n - b$ columns have at least one entry missing from $([m] \setminus i) \times \{j\}$. We see that $|S_1| - |A_i \cap S_1|$ of these columns can arise from elements of $S_1$, and $|S_2|$ of these columns can arise from elements of $S_2$. Therefore
$$n - b \le (|S_1| - |A_i \cap S_1|) + |S_2|.$$
Adding this to the equation $|A_i \cap S_1| + |S_2| > n-b$, we see that $2(n - b) < |S_1| + 2|S_2|$. But $2(n - b) = |S_1| + 2|S_2|$ because $|F| = 2b + na$. 
\end{proof}

The circuits which are not Laman circuits that we will use to prove Corollary~\ref{cor:laman} will be built out of Example~\ref{ex:nonlaman} using the following result. 

\begin{proposition}\label{prop:cone}
Let $S = \{(1, 1), \dotsc, (1, n)\}$. For any $p \ge 0$ and $0 < s < m$, the contraction $\mathrm{T}_{m,n}(s,r,p)/S$ is $\mathrm{T}_{m-1, n}(s-1, r, p)$, and the deletion $\mathrm{T}_{m,n}(s,r,p) \setminus S$ is $\mathrm{T}_{m-1, n}(s, r, p)$. 
\end{proposition}

\begin{proof}
Let $U, V$ be vector spaces over an infinite field $\mathbf{k}$ of characteristic $p$ of dimensions $s, r$ respectively. Choose generic vectors $u_1, \dotsc, u_m \in U$ and $v_1, \dotsc, v_n \in V$. 
The contraction $\mathrm{T}_{m,n}(s,r,p)/S$ is represented by the vector configuration $u_2 \otimes v_1, u_2 \otimes v_2, \dotsc, u_m \otimes v_n \in U/(\mathbf{k} \cdot u_1) \otimes V$, and the deletion $\mathrm{T}_{m,n}(s,r,p) \setminus S$ is represented by $u_2 \otimes v_1, u_2 \otimes v_2, \dotsc, u_m \otimes v_n \in U \otimes V$. 
\end{proof}

Using Proposition~\ref{prop:cone} and Theorem~\ref{thm:equivalence}(\ref{bipartitecase}), we obtain a simple proof of the ``cone lemma'' \cite{kalai2016bipartite}*{Lemma 3.12} for bipartite rigidity. 

\begin{corollary}\label{cor:cone}
Let $G$ be a bipartite graph on $[m] \times [n]$. Let $C_LG$ (respectively $C_RG$) be the bipartite graph on $[m + 1] \times [n]$ (resp. $[m] \times [n+1]$) obtained by adding a vertex to $G$ on the left (resp. right) and then connecting it to everything in $[n]$ (resp. $[m]$). Then $G$ is independent in $\mathcal{B}_{m,n}(a,b)$ if and only if $C_LG$ is independent in $\mathcal{B}_{m+1,n}(a+1, b)$ (resp. $C_RG$ is independent in $\mathcal{B}_{m, n+1}(a, b+1)$). 
\end{corollary}

\begin{proof}[Proof of Corollary~\ref{cor:laman}]
Suppose $2 \le a < m-2$ and $2 \le b < n-2$. Let $G$ be the graph described in Example~\ref{ex:nonlaman}, which is dependent in $\mathcal{B}_{5,5}(2,2)$ but does not contain a Laman circuit. If we perform $a-2$ left cones and $b - 2$ right cones, then we obtain a graph $\mathcal{H}$ which is dependent in $\mathcal{B}_{a+3, b+3}(a,b)$ but does not contain a Laman circuit. The restriction of $\mathcal{B}_{m,n}(a,b)$ to $[a + 3] \times [b+3]$ is $\mathcal{B}_{a+3, b+3}(a,b)$, e.g. by Proposition~\ref{prop:cone}, so $\mathcal{H}$ also gives a dependent set which does not contain a Laman circuit for $\mathcal{B}_{m,n}(a,b)$.

If $a \le 1$ or $b \le 1$, then the result follows from \cite[Theorem 5.4]{kalai2016bipartite}, which is based on \cite{whiteley1989matroid}.
If $m - a \le 2$ or $n - b \le 2$, then the result follows from Proposition~\ref{prop:Lamanbipartite}.
\end{proof}

\subsection{Challenges for $s=4$}
Consider $m = n = 7$ and $s = r = 4$. Consider the pattern
\[
    E = (\{1,2,3\}^2 \setminus \{(3,3)\}) \cup \{4,5\}^2 \cup \{6,7\}^2.
\]
This is illustrated in Figure~\ref{fig:7733}. Although $E$ is of basis size and its complement satisfies the Laman conditions, it is not a basis of the matroid $\mathrm{T}_{7,7}(4,4,p)$ for any $p$. 

\begin{figure}[ht]
\begin{center}
\begin{tikzpicture}
\newcommand{\cbb}{\Huge $\star$}
\newcommand{\cbr}{\Huge $\diamond$}

\fill[red!50!white] (0,7) rectangle (3,5);
\fill[red!50!white] (0,7) rectangle (2,4);
\fill[red!50!white] (3,4) rectangle (5,2);
\fill[red!50!white] (5,2) rectangle (7,0);

\fill[blue!40!white] (3,7) rectangle (7,4);
\fill[blue!40!white] (0,4) rectangle (3,0);
\fill[blue!40!white] (3,2) rectangle (5,0);
\fill[blue!40!white] (5,4) rectangle (7,2);

\draw (0.5,6.5) node {\cbr};
\draw (0.5,5.5) node {\cbr};
\draw (0.5,4.5) node {\cbr};
\draw (1.5,6.5) node {\cbr};
\draw (1.5,5.5) node {\cbr};
\draw (1.5,4.5) node {\cbr};
\draw (2.5,6.5) node {\cbr};
\draw (2.5,5.5) node {\cbr};

\draw (4.5,2.5) node {\cbr};
\draw (4.5,3.5) node {\cbr};
\draw (3.5,2.5) node {\cbr};
\draw (3.5,3.5) node {\cbr};

\draw (6.5,0.5) node {\cbr};
\draw (5.5,0.5) node {\cbr};
\draw (6.5,1.5) node {\cbr};
\draw (5.5,1.5) node {\cbr};

\draw (4.5,0.5) node {\cbb};
\draw (3.5,0.5) node {\cbb};
\draw (4.5,1.5) node {\cbb};
\draw (3.5,1.5) node {\cbb};

\draw (6.5,2.5) node {\cbb};
\draw (5.5,3.5) node {\cbb};
\draw (5.5,2.5) node {\cbb};
\draw (6.5,3.5) node {\cbb};

\draw (0.5,0.5) node {\cbb};
\draw (0.5,1.5) node {\cbb};
\draw (0.5,2.5) node {\cbb};
\draw (0.5,3.5) node {\cbb};
\draw (1.5,0.5) node {\cbb};
\draw (1.5,1.5) node {\cbb};
\draw (1.5,2.5) node {\cbb};
\draw (1.5,3.5) node {\cbb};
\draw (2.5,0.5) node {\cbb};
\draw (2.5,1.5) node {\cbb};
\draw (2.5,2.5) node {\cbb};
\draw (2.5,3.5) node {\cbb};

\draw (3.5,6.5) node {\cbb};
\draw (4.5,6.5) node {\cbb};
\draw (5.5,6.5) node {\cbb};
\draw (6.5,6.5) node {\cbb};
\draw (3.5,5.5) node {\cbb};
\draw (4.5,5.5) node {\cbb};
\draw (5.5,5.5) node {\cbb};
\draw (6.5,5.5) node {\cbb};
\draw (3.5,4.5) node {\cbb};
\draw (4.5,4.5) node {\cbb};
\draw (5.5,4.5) node {\cbb};
\draw (6.5,4.5) node {\cbb};

\draw (0,0) grid (7,7);
\end{tikzpicture}
\end{center}
\caption{(red $\diamond$) A circuit of $\mathrm{T}_{7,7}(4,4,p)$. (blue $\star$) The corresponding circuit in $\mathcal{B}_{7,7}(3,3)$. See Figure~\ref{fig:nonlaman} for how to interpret.}\label{fig:7733}
\end{figure}

\begin{proposition}\label{prop:7733}
    For all $p$, $E$ is not a basis of $\mathrm{T}_{7,7}(4,4,p)$.
\end{proposition}

\begin{proof}
    Let $u_1, \hdots, u_7, v_1, \hdots, v_7 \in \F^4$ be generic. Further define
    \begin{align*}
        W_0 &= \Span\{u_i \otimes v_j : (i,j) \in \{1,2,3\}^2\}\\
        W_1 &= \Span\{u_i \otimes v_j : (i,j) \in \{1,2,3\}^2 \setminus \{(3,3)\}\}\\
        W_2 &= \Span\{u_i \otimes v_j : (i,j) \in \{4,5\}^2\}\\
        W_3 &= \Span\{u_i \otimes v_j : (i,j) \in \{6,7\}^2\}.
    \end{align*}
    Assume for the sake of contradiction that $\dim(W_1 + W_2 + W_3) = 16.$ Note that there exists $w_2 \in W_0 \cap W_2$ and $w_3 \in W_0 \cap W_3$. Since $W_1, W_2$, and $W_3$ are linearly independent, then $W_1, w_2, w_3$ are linearly independent. This implies that $\dim(W_1 + \Span\{w_2\} + \Span\{w_3\}) = 10$, which contradicts the fact that $W_1 + \Span\{w_2\} + \Span\{w_3\} \subseteq W_0$.
\end{proof}

Note that the example of Proposition~\ref{prop:7733} is an obstruction to extending the proof of Proposition~\ref{lem:m-a=3} to $s=4$. In the proof of Proposition~\ref{lem:m-a=3}, once we reduce to the $S_3 = \emptyset$ case, we show (through six cases) that a set $E\subseteq [m]\times [n]$ is dependent if and only if there is a tensor product $R= L\otimes \F^r$ for some  $L\subseteq \F^s$ and a partition $E_1,E_2,\hdots,E_\ell$ of $E$ such that $\sum\limits_{i=1}^\ell \dim ((\Span E_i) \cap R) > r \dim L$. This gives a family of inequalities that are satisfied by independent sets, and by this method we obtain all the inequalities in Proposition~\ref{lem:m-a=3}. To prove the sufficiency of these inequalities, we use an inductive argument: when one such inequality is tight then we can quotient by the tensor subspace $L\otimes \F^r$ and work over the smaller space $(\F^s / L)\otimes \F^r$. In the example of Proposition~\ref{prop:7733}, the dependence is detected by looking at a subspace $L \otimes L$ in $\F^4\otimes \F^4$, where $\dim L = 3$. As $(\F^4\otimes \F^4) / (L \otimes L)$ is not naturally the tensor product of two spaces, the simple inductive argument of Proposition~\ref{lem:m-a=3} will not work. 
\section{Characteristic independence}

In this section, we prove Theorem~\ref{thm:char}. The cases $s=0$ and $m = s$ are trivial. The descriptions of $\mathrm{T}_{m,n}(s,r,p)$ when $s \le 3$ (Section~\ref{sec:m-asmall}) or when $m - s  = 1$ (\cite{gopalan2017maximally}) are independent of the characteristic, so it remains to do the case when $m - s = n - r = 2$. For this we use Bernstein's description of the independent sets of $\mathcal{B}_{m,n}(2,2)$, Theorem~\ref{thm:Bernstein},  which gives a description of $\mathrm{T}_{m,n}(m-2,n-2,0)$. The following result shows that the dependent sets of $\mathrm{T}_{m,n}(m-2,n-2,0)$ are also dependent in $\mathrm{T}_{m,n}(m-2, n-2, p)$, so we only need to show that the bases of $\mathrm{T}_{m,n}(m-2, n-2, 0)$ are bases of $\mathrm{T}_{m,n}(m-2, n-2, p)$. 

\begin{proposition}\label{prop:uppersemi}
Any independent set of $\mathrm{T}_{m,n}(s,r, p)$ is an independent set of $\mathrm{T}_{m,n}(s,r,0)$. 
\end{proposition}

\begin{proof}
The proof of Theorem~\ref{thm:equivalence} shows that one can check if a set is independent in $\mathrm{T}_{m,n}(s,r,0)$ in terms of the rank of a matrix whose entries are polynomials with integer coefficients in $\{x_{ij}, y_{k\ell}\}$. We can check if a set is independent in $\mathrm{T}_{m,n}(s,r,p)$ by taking the same matrix and computing the rank over a field of characteristic $p$. 
\end{proof}

We now show that the independent sets of $\mathcal{B}_{m,n}(2,2)$, as described in Theorem~\ref{thm:Bernstein}, are still independent in the dual of $\mathrm{T}_{m,n}(m-2,n-2,p)$ for any $p$. 

\begin{proposition}\label{prop:char-free-Bernstein}
    If a bipartite graph has an edge orientation with no directed cycles or alternating cycles, then the graph is independent in the dual of $\mathrm{T}_{m,n}(m-2,n-2,p)$ for any $p$.
\end{proposition}

\begin{proof}
Let $G$ be a bipartite graph on $[m] \sqcup [n]$ which has an edge orientation with no directed cycles or alternating cycles. Let $M_G$ be the $2|V(G)| \times |E(G)|$ matrix obtained by taking the columns of the matrix in Definition~\ref{def:bipartiterigidity} (with $a = b = 2$) indexed by edges of $G$. Instead of using the variables $\{x_{ij}, y_{k\ell}\}_{(i, j) \in [m] \times [2], \, (k, \ell) \in [n] \times [2]}$, it will be convenient to  use the variables $\{x_{vc}\}_{v \in V(G), \, c \in \{1, 2\}}$. Set $2V(G)=\{x_{vc}:v\in V(G),c\in\{1,2\}\}$ to be the set of variables. We also index the rows of $M_G$ by $2V(G)$. We will show that there is a maximal minor of $M_G$ for which some monomial occurs with coefficient $\pm 1$, so $M_G$ has the same rank in any characteristic. The proof of Theorem~\ref{thm:equivalence}(\ref{bipartitecase}) then implies the result.

For each injective map $\sigma \colon E(G) \to 2V(G)$, we set $M_\sigma \coloneqq \prod_{e\in E(G)}(M_G)_{\sigma(e),e}$. Up to a sign, this is one term in the expansion of a maximal minor of $M_G$. 
Since the column of $M_G$ corresponding to an edge $e=(u,v)$ has only $4$ non-zero entries $x_{u1},x_{u2},x_{v1}$, and $x_{v2}$, there are only a few $M_\sigma$ that are non-zero. A non-zero $M_\sigma$ can be represented by a $2$-colored directed version of $G$ (denoted $G_\sigma$): if $\sigma(e)=x_{uc}$, we color $(u,v)$ with color $c$ and direct $u\to v$. Note that $(M_G)_{x_{uc},e}=x_{vc}$.
Set $\text{indeg}_{c,G_\sigma}(v)$ to be the number of edges directed towards $v$ in $G_\sigma$ of color $c$, and similarly define $\text{outdeg}_{c, G_{\sigma}}$. 
In order for the map $\sigma$ to be injective, the outdegree of each vertex in $G_\sigma$ must be at most $1$ per color. In fact $\text{outdeg}_{c,G_\sigma}(v)=1$ if $x_{vc}\in \sigma(E(G))$ and is $0$ otherwise.  We have
    \[M_\sigma= \prod_{v\in V(G),\, c\in\{1,2\}} x_{vc}^{\text{indeg}_{c,G_\sigma}(v)}.\]

Using the acyclic orientation of $G$ with no alternating cycles, we construct a $G_{\sigma}$ for which the corresponding monomial $M_{\sigma}$ does not occur for any other injective map $\sigma' \colon E(G) \to 2V(G)$. We color the edges of $G$ with color $1$ if they are oriented from $[m]$ to $[n]$, and we use color $2$ if they are oriented from $[n]$ to $[m]$. This coloring has no monochromatic cycles or cycles which alternate in color. Now we pick out a root for every monochromatic connected tree, and we direct the edges towards the root. In this way, we obtain a directed $2$-coloring of $G$, say $G_\sigma$. As $\text{outdeg}_{c,G_{\sigma}}(v) \le 1$ for each $v$ and $c$, this $G_{\sigma}$ does indeed arise from an injective map $\sigma \colon E(G) \to 2V(G)$. Note that directions on the edges are \emph{not} the same as the orientation we used to construct $G_{\sigma}$. 

We claim that the monomial $M_\sigma$ only appears once in the determinant of the minor of $M_G$ with rows indexed by $\sigma(E(G))$, which is a maximal minor of $M_G$. 
Otherwise, there is another $G_{\sigma'}$ with the same monomial weight. Therefore $$\text{indeg}_{c,G_\sigma}(v)=\text{indeg}_{c,G_{\sigma'}}(v),\text{ for any $v\in V(G)$ and each color $c\in\{1,2\}$}.$$
Note that $\text{outdeg}_{c,G_\sigma}(v)=\text{outdeg}_{c,G_{\sigma'}}(v)$ for all $c,v$ since $\sigma(E(G))=\sigma'(E(G))$ as they appear in the same maximal minor and the range of $\sigma$ determines which vertices have outdegree $1$ for each color. Therefore $$\text{deg}_{c,G_\sigma}(v)=\text{deg}_{c,G_{\sigma'}}(v),\text{ for any $v\in V(G)$ and color $c\in\{1,2\}$}.$$
Note that $G_\sigma$ and $G_{\sigma'}$ have to give different colors to at least one edge. Indeed, suppose they give the same color to each edge. Since $G_\sigma$ has no monochromatic cycles, every monochromatic strongly connected component of $G_\sigma$ is a tree. Each tree has a unique vertex with $\text{outdeg}_{c,G_\sigma}=0$. We make that vertex the root of the tree and direct all edges of the tree towards the root. This recovers the directions for the edges of $G_\sigma$, as this is how $G_\sigma$ was defined. But the same process can be followed for $G_{\sigma'}$ to arrive at the same directions for the edges, i.e., once the outdegrees for each color are fixed, there is only one way to direct the graph. This implies that $\sigma=\sigma'$, so $G_\sigma$ and $G_{\sigma'}$ have to differ in at least one color if they are distinct.

However, consider the following process. We start with a vertex $v_0$ that is adjacent to an edge $e=(v_0,v_1)$ which is given a different color in $G_\sigma$ and $G_{\sigma'}$. We can assume $e$ has color $1$ in $G_\sigma$ and color $2$ in $G_{\sigma'}$. Since $\text{deg}_{c,G_\sigma}(v_1)=\text{deg}_{c,G_{\sigma'}}(v_1)$, there must be another edge $e'=(v_1,v_2)$ adjacent to $v_1$ that has color $2$ in $G_\sigma$ and color $1$ in $G_{\sigma'}$. Now we can use the same argument repeatedly to find a path $v_0\to v_1\to v_2\to\cdots$ which eventually self-intersects and forms a loop. However, this loop will be a cycle which alternates in color in $G_\sigma$, but there are no such cycles in $G_{\sigma}$.
\end{proof}

\begin{example}
Consider the graph on the left of Figure~\ref{fig:example}, with the indicated edge orientation, which has no directed cycles or alternating cycles. After constructing the corresponding coloring with no monochromatic or alternating cycles, each color is a tree. We choose $v_3$ as a root of the red tree and $u_1$ as a root of the blue tree. If red is the first color and blue is the second color, then the corresponding monomial $M_{\sigma}$ is $x_{u_1,1} x_{v_1,1} x_{v_3,1}^2 x_{u_1,2} x_{u_2, 2}x_{v_2,2}^2$, and $\sigma$ is the map given by $\sigma(\{u_1, v_1\}) = x_{v_1, 1}$, $\sigma(\{u_1, v_2\}) = x_{v_2, 2}$, $\sigma(\{u_1, v_3\}) = x_{u_1, 1}$, $\sigma(\{u_2, v_1\}) = x_{v_1, 2}$, $\sigma(\{u_2, v_2\}) = x_{u_2, 2}$, $\sigma(\{u_2, v_3\}) = x_{u_2, 1}$, $\sigma(\{u_3, v_1\}) = x_{u_3, 1}$, and $\sigma(\{u_3, v_2\}) = x_{u_3, 2}$. 
\end{example}

\begin{figure}
\begin{tikzpicture}[
    vertex/.style={
        circle,
        fill=black,
        draw=none,
        inner sep=0pt,
        minimum size=5pt
    },
    directed edge/.style={
        ->,
        >={Stealth[round, length=3mm, width=2mm]},
        thick,
        draw=gray!70
    }
]

\foreach \i in {1,2,3} {
    \node[vertex, label=left:{$u_\i$}] (u\i) at (0, {2 - (\i-1)*2}) {};
}

\foreach \j in {1,2,3} {
    \node[vertex, label=right:{$v_\j$}] (v\j) at (4, {2 - (\j-1)*2}) {};
}

\draw[directed edge] (u1) -- (v1);
\draw[directed edge] (v2) -- (u1);
\draw[directed edge] (u1) -- (v3);
\draw[directed edge] (v1) -- (u2);
\draw[directed edge] (v2) -- (u2);
\draw[directed edge] (u2) -- (v3);
\draw[directed edge] (u3) -- (v1);
\draw[directed edge] (v2) -- (u3);

\end{tikzpicture}
\begin{tikzpicture}[
    vertex/.style={
        circle,
        fill=black,
        draw=none,
        inner sep=0pt,
        minimum size=5pt
    },
    directed edge/.style={
        ->,
        >={Stealth[round, length=3mm, width=2mm]},
        thick,
        draw=gray!70
    }
]

\foreach \i in {1,2,3} {
    \node[vertex, label=left:{$u_\i$}] (u\i) at (0, {2 - (\i-1)*2}) {};
}

\foreach \j in {1,2,3} {
    \node[vertex, label=right:{$v_\j$}] (v\j) at (4, {2 - (\j-1)*2}) {};
}

\draw[directed edge] [color=red] (v1) -- (u1);
\draw[directed edge] [color=blue] (v2) -- (u1);
\draw[directed edge] [color=red] (u1) -- (v3);
\draw[directed edge] [color=blue] (v1) -- (u2);
\draw[directed edge] [color=blue] (u2) -- (v2);
\draw[directed edge] [color = red] (u2) -- (v3);
\draw[directed edge] [color=red] (u3) -- (v1);
\draw[directed edge] [color=blue] (u3) -- (v2);

\end{tikzpicture}

\caption{A bipartite graph $G$ with an edge orientation which has no directed cycles or alternating cycles, and the corresponding $G_{\sigma}$.}\label{fig:example}
\end{figure}

\begin{remark}
    The argument above can also be applied to rank $2$ skew-symmetric matrix completion. 
    A $2$-coloring of the edges of a graph $G$ is \emph{unbalanced} if it has no monochromatic cycles or trails which alternate in color. An acyclic orientation of the edges of $G$ is \emph{unbalanced} if it has no alternating trail. 
    When $G$ is bipartite, an unbalanced coloring is equivalent to an unbalanced acyclic orientation, so Theorem~\ref{thm:Bernstein} states that $G$ is independent in $\mathcal{B}_{m,n}(2,2)$ if and only if it has an unbalanced coloring.

    In \cite{bernstein2017completion}, Bernstein proved that a graph $G$ has an unbalanced acyclic orientation if and only if it is independent in $\mathcal{H}_{n}(2)$. Extending the proof of Proposition~\ref{prop:char-free-Bernstein}, one can show that a graph with an unbalanced coloring is independent in the dual of $\mathrm{W}_n(n-2, p)$ for any $p$. This implies that a graph with an unbalanced coloring has an unbalanced acyclic orientation, but we do not know a combinatorial proof of this fact. We do not know if the converse holds.
    
\end{remark}

\begin{proof}[Proof of Theorem~\ref{thm:char}]
The case when $s = 0$ or $s = m$ is trivial. The case when $m - s = 1$ is proven in \cite{gopalan2017maximally}. The case when $1 \le s \le  3$ is proven in Corollary~\ref{lem:r=1}, Proposition~\ref{lem:m-a=2}, and Proposition~\ref{lem:m-a=3}. The case when $m - s = n - r = 2$ is proven in Proposition~\ref{prop:char-free-Bernstein}. 
\end{proof}

\section{Conjectural description of the bipartite rigidity matroid}

We now give a conjectural description of the independent sets of $\mathcal{B}_{m,n}(d,d)$ for all $d$. Using Proposition~\ref{prop:cone}, this gives a description of the independent sets of $\mathcal{B}_{m,n}(a,b)$ for all $a$ and $b$. Our conjecture is inspired by Bernstein's proof of Theorem~\ref{thm:Bernstein} using tropical geometry \cite{bernstein2017completion}. We show that the sets we describe are in fact independent in $\mathcal{B}_{m,n}(d,d)$, and moreover are independent in the dual of $\mathrm{T}_{m,n}(m-d, n-d, p)$ for all $p$. In particular, our conjecture implies that the tensor matroid is independent of the characteristic.

\begin{definition}\label{def:d-bernstein}
    Given a bipartite graph $G$ and an integer $d\geq 0$, a $d$-coloring of the edges of $G$ is \emph{$d$-Bernstein} if there are no monochromatic cycles and there exists a labeling $c \colon V(G)\to \mathbb{R}^d$ which sends $v\mapsto (c_1(v),c_2(v),\dots,c_d(v))$ that satisfies the following conditions:
    \begin{enumerate}
        \item $c_1(v)+c_2(v)+\cdots+c_d(v)=0$ for any $v\in V(G)$;
        \item for every edge $(u,v)\in E(G)$ with color $i$, $c_i(u)+c_i(v)>c_j(u)+c_j(v)$ for any $j\in[d]\setminus\{i\}$.
    \end{enumerate}
\end{definition}

\begin{remark}
    When $d=2$, we recover Bernstein's condition in Theorem~\ref{thm:Bernstein}, if we orient the edges with color $1$ from $[m]$ to $[n]$ and orient the edges with color $2$ from $[n]$ to $[m]$.
\end{remark}

\begin{proposition}\label{prop:d-Bernstein}
    If a bipartite graph $G$ admits a $d$-coloring that is $d$-Bernstein, then $G$ is independent in the dual of $\mathrm{T}_{m,n}(m-d,n-d,p)$ for all $p$.
\end{proposition}

\begin{proof}
    The proof is almost identical to the proof of Proposition~\ref{prop:char-free-Bernstein} except in the last step. The statement reduces to proving that if there is a $d$-Bernstein coloring of $G$ say $G_\sigma$, then there does not exist another $d$-coloring of $G$ say $G_{\sigma'}$, where $G_\sigma$ and $G_{\sigma'}$ differ in at least one edge color, and
    $$\text{deg}_{c,G_\sigma}(v)=\text{deg}_{c,G_{\sigma'}}(v),\text{ for any $v\in V(G)$ and color $c\in[d]$}.$$
    
    Consider the following equality:
    \[\sum_{v\in V(G),i\in[d]}\text{deg}_{c,G_\sigma}(v)c_i(v)=\sum_{e=(u,v)\in E(G)}(c_{\sigma(e)}(u)+c_{\sigma(e)}(v)),\text{ where }\sigma(e)\text{ is the color of $e$ in $G_\sigma$}.\]
    This equality also holds when we replace $\sigma$ with $\sigma'$. However, as we change from $\sigma$ to $\sigma'$, the left hand side does not change, but the right hand side strictly decreases because of condition (2) in Definition~\ref{def:d-bernstein} and $\sigma,\sigma'$ give different colors to at least one edge.
\end{proof}

\begin{conjecture}\label{conj:generalizedbernstein}
    If $G$ is independent in $\mathcal{B}_{m,n}(d,d)$, then $G$ admits a $d$-coloring that is $d$-Bernstein.
\end{conjecture}
When $d = 1$, this holds because $\mathcal{B}_{m,n}(1,1)$ is the graphical matroid of the complete bipartite graph $\mathrm{K}_{m,n}$. 
When $d=2$, this is proven in \cite{bernstein2017completion}. We have checked this conjecture for $\mathcal{B}_{5,5}(3,3)$ and for $\mathcal{B}_{m,n}(d,d)$ when $d \ge m-1$. This conjecture and Proposition~\ref{prop:cone} imply that $\mathrm{T}_{m,n}(s,r,0)=\mathrm{T}_{m,n}(s,r,p)$ for all $p$.

\begin{remark}
    The $d$-Bernstein condition is closely related to the notion of \emph{Barvinok rank} for tropical matrices \cite{tropicalrank}. Conjecture~\ref{conj:generalizedbernstein} is equivalent to saying that the Barvinok rank $d$ cones, which are a subset of the cones in the tropical determinantal variety, determine the matroid $\mathcal{B}_{m,n}(d,d)$.
\end{remark}

\bibliographystyle{alpha}
\bibliography{ref.bib}

\begin{bibdiv}
\begin{biblist}

\bib{athi2023structure}{inproceedings}{
      author={Athi, Harshithanjani},
      author={Chigullapally, Rasagna},
      author={Krishnan, Prasad},
      author={Lalitha, V},
       title={On the structure of higher order {MDS} codes},
organization={IEEE},
        date={2023},
   booktitle={2023 {IEEE} {I}nt. {S}ymp. {I}nform. {T}heory ({ISIT})},
       pages={1009\ndash 1014},
}

\bib{alrabiah2023ag}{article}{
      author={Alrabiah, Omar},
      author={Guruswami, Venkatesan},
      author={Li, Ray},
       title={A{G} codes have no list-decoding friends: approaching the
  generalized {S}ingleton bound requires exponential alphabets},
        date={2024},
       pages={1367\ndash 1378},
         url={https://doi.org/10.1137/1.9781611977912.55},
      review={\MR{4699300}},
}

\bib{alrabiah2023randomly}{inproceedings}{
      author={Alrabiah, Omar},
      author={Guruswami, Venkatesan},
      author={Li, Ray},
       title={Randomly punctured {R}eed-{S}olomon codes achieve list-decoding
  capacity over linear-sized fields},
        date={2024},
   booktitle={S{TOC}'24---{P}roceedings of the 56th {A}nnual {ACM} {S}ymposium
  on {T}heory of {C}omputing},
   publisher={ACM, New York},
       pages={1458\ndash 1469},
         url={https://doi.org/10.1145/3618260.3649634},
      review={\MR{4764921}},
}

\bib{Anderson}{article}{
      author={Anderson, Nicholas},
       title={Matroid products in tropical geometry},
        date={2024},
        ISSN={2522-0144,2197-9847},
     journal={Res. Math. Sci.},
      volume={11},
      number={2},
       pages={Paper No. 38, 28},
         url={https://doi.org/10.1007/s40687-024-00452-z},
      review={\MR{4741183}},
}

\bib{BernsteinBlekhermanLee}{article}{
      author={Bernstein, Daniel~Irving},
      author={Blekherman, Grigoriy},
      author={Lee, Kisun},
       title={Typical ranks in symmetric matrix completion},
        date={2021},
        ISSN={0022-4049,1873-1376},
     journal={J. Pure Appl. Algebra},
      volume={225},
      number={7},
       pages={Paper No. 106603, 16},
         url={https://doi.org/10.1016/j.jpaa.2020.106603},
      review={\MR{4217924}},
}

\bib{brakensiek2023improved}{inproceedings}{
      author={Brakensiek, Joshua},
      author={Dhar, Manik},
      author={Gopi, Sivakanth},
       title={Improved field size bounds for higher order {MDS} codes},
organization={IEEE},
        date={2023},
   booktitle={2023 {IEEE} {I}nt. {S}ymp. {I}nform. {T}heory ({ISIT})},
       pages={1243\ndash 1248},
}

\bib{brakensiek2023generalized}{article}{
      author={Brakensiek, Joshua},
      author={Dhar, Manik},
      author={Gopi, Sivakanth},
       title={Generalized {GM}-{MDS}: polynomial codes are higher order {MDS}},
        date={2024},
       pages={728\ndash 739},
         url={https://doi.org/10.1145/3618260.3649637},
      review={\MR{4764853}},
}

\bib{bdgz2023}{article}{
      author={Brakensiek, Joshua},
      author={Dhar, Manik},
      author={Gopi, Sivakanth},
      author={Zhang, Zihan},
       title={A{G} codes achieve list-decoding capacity over constant-sized
  fields},
        date={2025},
        ISSN={0018-9448,1557-9654},
     journal={IEEE Trans. Inform. Theory},
      volume={71},
      number={8},
       pages={5935\ndash 5956},
         url={https://doi.org/10.1109/tit.2025.3577506},
      review={\MR{4937538}},
}

\bib{bell2023kapranov}{article}{
      author={Brakensiek, Joshua},
      author={Eur, Christopher},
      author={Larson, Matt},
      author={Li, Shiyue},
       title={Kapranov {D}egrees},
        date={2025},
        ISSN={1073-7928,1687-0247},
     journal={Int. Math. Res. Not. IMRN},
      number={20},
       pages={rnaf306},
         url={https://doi-org.ezproxy.princeton.edu/10.1093/imrn/rnaf306},
      review={\MR{4970179}},
}

\bib{bernstein2017completion}{article}{
      author={Bernstein, Daniel~Irving},
       title={Completion of tree metrics and rank 2 matrices},
        date={2017},
        ISSN={0024-3795,1873-1856},
     journal={Linear Algebra Appl.},
      volume={533},
       pages={1\ndash 13},
         url={https://doi.org/10.1016/j.laa.2017.07.016},
      review={\MR{3695897}},
}

\bib{bgm2022mds}{article}{
      author={Brakensiek, Joshua},
      author={Gopi, Sivakanth},
      author={Makam, Visu},
       title={Lower bounds for maximally recoverable tensor codes and higher
  order {MDS} codes},
        date={2022},
        ISSN={0018-9448,1557-9654},
     journal={IEEE Trans. Inform. Theory},
      volume={68},
      number={11},
       pages={7125\ndash 7140},
      review={\MR{4524629}},
}

\bib{bgm2023generic}{article}{
      author={Brakensiek, Joshua},
      author={Gopi, Sivakanth},
      author={Makam, Visu},
       title={Generic {R}eed-{S}olomon codes achieve list-decoding capacity},
        date={2023},
       pages={1488\ndash 1501},
         url={https://doi.org/10.1145/3564246.3585128},
      review={\MR{4617483}},
}

\bib{BlekhermanSinn}{article}{
      author={Blekherman, Grigoriy},
      author={Sinn, Rainer},
       title={Maximum likelihood threshold and generic completion rank of
  graphs},
        date={2019},
        ISSN={0179-5376,1432-0444},
     journal={Discrete Comput. Geom.},
      volume={61},
      number={2},
       pages={303\ndash 324},
         url={https://doi.org/10.1007/s00454-018-9990-3},
      review={\MR{3903791}},
}

\bib{JacksonRigidity}{article}{
      author={Clinch, Katie},
      author={Jackson, Bill},
      author={Tanigawa, Shin-ichi},
       title={Abstract 3-rigidity and bivariate {$C^1_2$}-splines {I}:
  {W}hiteley's maximality conjecture},
        date={2022},
        ISSN={2397-3129},
     journal={Discrete Anal.},
       pages={Paper No. 2, 50},
         url={https://doi.org/10.19086/da},
      review={\MR{4411108}},
}

\bib{CJT2}{article}{
      author={Clinch, Katie},
      author={Jackson, Bill},
      author={Tanigawa, Shin-ichi},
       title={Abstract 3-rigidity and bivariate {$C^1_2$}-splines {II}:
  {C}ombinatorial characterization},
        date={2022},
        ISSN={2397-3129},
     journal={Discrete Anal.},
       pages={Paper No. 3, 32},
         url={https://doi.org/10.19086/da},
      review={\MR{4450654}},
}

\bib{CLY}{article}{
      author={Cai, May},
      author={Lee, Kisun},
      author={Yu, Josephine},
       title={The tropical variety of symmetric rank 2 matrices},
        date={2025},
        ISSN={0024-3795,1873-1856},
     journal={Linear Algebra Appl.},
      volume={720},
       pages={50\ndash 71},
         url={https://doi.org/10.1016/j.laa.2025.04.011},
      review={\MR{4900135}},
}

\bib{RuizBipartite}{article}{
      author={Crespo~Ruiz, Luis},
       title={Realizations of multiassociahedra via bipartite rigidity},
        date={2023},
        note={ar{X}iv:2303.15776},
}

\bib{CrespoRuizSantos}{article}{
      author={Crespo~Ruiz, Luis},
      author={Santos, Francisco},
       title={Bar-and-joint rigidity on the moment curve coincides with
  cofactor rigidity on a conic},
        date={2023},
        ISSN={2766-1334},
     journal={Comb. Theory},
      volume={3},
      number={1},
       pages={Paper No. 15, 13},
      review={\MR{4565302}},
}

\bib{RuizSantos1}{article}{
      author={Crespo~Ruiz, Luis},
      author={Santos, Francisco},
       title={Multitriangulations and tropical {P}faffians},
        date={2024},
        ISSN={2470-6566},
     journal={SIAM J. Appl. Algebra Geom.},
      volume={8},
      number={2},
       pages={302\ndash 332},
         url={https://doi.org/10.1137/22M1527507},
      review={\MR{4736289}},
}

\bib{RuizSantos2}{article}{
      author={Crespo~Ruiz, Luis},
      author={Santos, Francisco},
       title={Realizations of multiassociahedra via rigidity},
        date={2025},
        ISSN={0179-5376,1432-0444},
     journal={Discrete Comput. Geom.},
      volume={73},
      number={4},
       pages={973\ndash 1015},
         url={https://doi.org/10.1007/s00454-024-00698-y},
      review={\MR{4903706}},
}

\bib{DraismaRincon}{article}{
      author={Draisma, Jan},
      author={Rinc\'{o}n, Felipe},
       title={Tropical ideals do not realise all {B}ergman fans},
        date={2021},
        ISSN={2522-0144,2197-9847},
     journal={Res. Math. Sci.},
      volume={8},
      number={3},
       pages={Paper No. 44, 11},
         url={https://doi.org/10.1007/s40687-021-00271-6},
      review={\MR{4280566}},
}

\bib{tropicalrank}{incollection}{
      author={Develin, Mike},
      author={Santos, Francisco},
      author={Sturmfels, Bernd},
       title={On the rank of a tropical matrix},
        date={2005},
   booktitle={Combinatorial and computational geometry},
      series={Math. Sci. Res. Inst. Publ.},
      volume={52},
   publisher={Cambridge Univ. Press, Cambridge},
       pages={213\ndash 242},
      review={\MR{2178322}},
}

\bib{dau2014gmmds}{inproceedings}{
      author={Dau, Son~Hoang},
      author={Song, Wentu},
      author={Yuen, Chau},
       title={On the existence of {MDS} codes over small fields with
  constrained generator matrices},
organization={IEEE},
        date={2014},
   booktitle={2014 {IEEE} {I}nternational {S}ymposium on {I}nformation
  {T}heory},
       pages={1787\ndash 1791},
}

\bib{gopalan2017maximally}{inproceedings}{
      author={Gopalan, Parikshit},
      author={Hu, Guangda},
      author={Kopparty, Swastik},
      author={Saraf, Shubhangi},
      author={Wang, Carol},
      author={Yekhanin, Sergey},
       title={Maximally recoverable codes for grid-like topologies},
        date={2017},
   booktitle={Proceedings of the {T}wenty-{E}ighth {A}nnual {ACM}-{SIAM}
  {S}ymposium on {D}iscrete {A}lgorithms},
   publisher={SIAM, Philadelphia, PA},
       pages={2092\ndash 2108},
         url={https://doi.org/10.1137/1.9781611974782.136},
      review={\MR{3627865}},
}

\bib{Graver}{article}{
      author={Graver, Jack~E.},
       title={Rigidity matroids},
        date={1991},
        ISSN={0895-4801},
     journal={SIAM J. Discrete Math.},
      volume={4},
      number={3},
       pages={355\ndash 368},
         url={https://doi.org/10.1137/0404032},
      review={\MR{1105942}},
}

\bib{GrossSullivant}{article}{
      author={Gross, Elizabeth},
      author={Sullivant, Seth},
       title={The maximum likelihood threshold of a graph},
        date={2018},
        ISSN={1350-7265,1573-9759},
     journal={Bernoulli},
      volume={24},
      number={1},
       pages={386\ndash 407},
         url={https://doi.org/10.3150/16-BEJ881},
      review={\MR{3706762}},
}

\bib{guo2023randomly}{article}{
      author={Guo, Zeyu},
      author={Zhang, Zihan},
       title={Randomly punctured {R}eed-{S}olomon codes achieve the list
  decoding capacity over polynomial-size alphabets},
        date={2023},
       pages={164\ndash 176},
         url={https://doi.org/10.1109/FOCS57990.2023.00019},
      review={\MR{4720261}},
}

\bib{MatroidComplexity}{article}{
      author={Hausmann, D.},
      author={Korte, B.},
       title={Algorithmic versus axiomatic definitions of matroids},
        date={1981},
        ISSN={0303-3929},
     journal={Math. Programming Stud.},
      number={14},
       pages={98\ndash 111},
         url={https://doi.org/10.1007/bfb0120924},
      review={\MR{600125}},
}

\bib{holzbaur2021correctable}{inproceedings}{
      author={Holzbaur, Lukas},
      author={Puchinger, Sven},
      author={Yaakobi, Eitan},
      author={Wachter-Zeh, Antonia},
       title={Correctable erasure patterns in product topologies},
organization={IEEE},
        date={2021},
   booktitle={2021 {IEEE} {I}nt. {S}ymp. {I}nform. {T}heory ({ISIT})},
       pages={2054\ndash 2059},
}

\bib{JacksonTanigawa}{article}{
      author={Jackson, Bill},
      author={Tanigawa, Shin-ichi},
       title={Maximal matroids in weak order posets},
        date={2024},
        ISSN={0095-8956,1096-0902},
     journal={J. Combin. Theory Ser. B},
      volume={165},
       pages={20\ndash 46},
         url={https://doi.org/10.1016/j.jctb.2023.10.012},
      review={\MR{4668802}},
}

\bib{KalaiShifting}{incollection}{
      author={Kalai, Gil},
       title={Algebraic shifting},
        date={2002},
   booktitle={Computational commutative algebra and combinatorics ({O}saka,
  1999)},
      series={Adv. Stud. Pure Math.},
      volume={33},
   publisher={Math. Soc. Japan, Tokyo},
       pages={121\ndash 163},
         url={https://doi.org/10.2969/aspm/03310121},
      review={\MR{1890098}},
}

\bib{KalaiHyperconnectivity}{article}{
      author={Kalai, Gil},
       title={Hyperconnectivity of graphs},
        date={1985},
        ISSN={0911-0119,1435-5914},
     journal={Graphs Combin.},
      volume={1},
      number={1},
       pages={65\ndash 79},
         url={https://doi.org/10.1007/BF02582930},
      review={\MR{796184}},
}

\bib{kong2021new}{article}{
      author={Kong, Xiangliang},
      author={Ma, Jingxue},
      author={Ge, Gennian},
       title={New bounds on the field size for maximally recoverable codes
  instantiating grid-like topologies},
        date={2021},
        ISSN={0925-9899,1572-9192},
     journal={J. Algebraic Combin.},
      volume={54},
      number={2},
       pages={529\ndash 557},
         url={https://doi.org/10.1007/s10801-021-01013-1},
      review={\MR{4322534}},
}

\bib{kalai2016bipartite}{article}{
      author={Kalai, Gil},
      author={Nevo, Eran},
      author={Novik, Isabella},
       title={Bipartite rigidity},
        date={2016},
        ISSN={0002-9947,1088-6850},
     journal={Trans. Amer. Math. Soc.},
      volume={368},
      number={8},
       pages={5515\ndash 5545},
         url={https://doi.org/10.1090/tran/6512},
      review={\MR{3458389}},
}

\bib{KRT}{article}{
      author={Kir\'{a}ly, Franz},
      author={Rosen, Zvi},
      author={Theran, Louis},
       title={Algebraic matroids with graph symmetry},
        date={2013},
        note={ar{X}iv:1312.3777},
}

\bib{Laman}{article}{
      author={Laman, G.},
       title={On graphs and rigidity of plane skeletal structures},
        date={1970},
        ISSN={0022-0833,1573-2703},
     journal={J. Engrg. Math.},
      volume={4},
       pages={331\ndash 340},
         url={https://doi.org/10.1007/BF01534980},
      review={\MR{269535}},
}

\bib{Lang}{book}{
      author={Lang, Serge},
       title={Algebra},
     edition={third},
      series={Graduate Texts in Mathematics},
   publisher={Springer-Verlag, New York},
        date={2002},
      volume={211},
        ISBN={0-387-95385-X},
         url={https://doi-org.ezproxy.princeton.edu/10.1007/978-1-4613-0041-0},
      review={\MR{1878556}},
}

\bib{lovett2021sparse}{article}{
      author={Lovett, Shachar},
       title={Sparse {MDS} matrices over small fields: a proof of the
  {GM}-{MDS} conjecture},
        date={2021},
        ISSN={0097-5397,1095-7111},
     journal={SIAM J. Comput.},
      volume={50},
      number={4},
       pages={1248\ndash 1262},
         url={https://doi.org/10.1137/20M1323345},
      review={\MR{4290098}},
}

\bib{Lovasz}{incollection}{
      author={Lov\'{a}sz, L.},
       title={Flats in matroids and geometric graphs},
        date={1977},
   booktitle={Combinatorial surveys ({P}roc. {S}ixth {B}ritish {C}ombinatorial
  {C}onf., {R}oyal {H}olloway {C}oll., {E}gham, 1977)},
   publisher={Academic Press, London-New York},
       pages={45\ndash 86},
      review={\MR{480111}},
}

\bib{las1981products}{article}{
      author={Las~Vergnas, Michel},
       title={On products of matroids},
        date={1981},
        ISSN={0012-365X,1872-681X},
     journal={Discrete Math.},
      volume={36},
      number={1},
       pages={49\ndash 55},
         url={https://doi.org/10.1016/0012-365X(81)90172-2},
      review={\MR{621892}},
}

\bib{liu2023linearized}{inproceedings}{
      author={Liu, Hedongliang},
      author={Wei, Hengjia},
      author={Wachter-Zeh, Antonia},
      author={Schwartz, Moshe},
       title={Linearized {R}eed--{S}olomon codes with support-constrained
  generator matrix},
organization={IEEE},
        date={2023},
   booktitle={2023 {IEEE} {I}nformation {T}heory {W}orkshop ({ITW})},
       pages={7\ndash 12},
}

\bib{LovaszYemini}{article}{
      author={Lov\'{a}sz, L.},
      author={Yemini, Y.},
       title={On generic rigidity in the plane},
        date={1982},
        ISSN={0196-5212},
     journal={SIAM J. Algebraic Discrete Methods},
      volume={3},
      number={1},
       pages={91\ndash 98},
         url={https://doi.org/10.1137/0603009},
      review={\MR{644960}},
}

\bib{Mason}{incollection}{
      author={Mason, J.~H.},
       title={Glueing matroids together: a study of {D}ilworth truncations and
  matroid analogues of exterior and symmetric powers},
        date={1981},
   booktitle={Algebraic methods in graph theory, {V}ol. {I}, {II} ({S}zeged,
  1978)},
      series={Colloq. Math. Soc. J\'{a}nos Bolyai},
      volume={25},
   publisher={North-Holland, Amsterdam-New York},
       pages={519\ndash 561},
      review={\MR{642060}},
}

\bib{Matsumura}{book}{
      author={Matsumura, Hideyuki},
       title={Commutative ring theory},
      series={Cambridge Studies in Advanced Mathematics},
   publisher={Cambridge University Press, Cambridge},
        date={1986},
      volume={8},
        ISBN={0-521-25916-9},
        note={Translated from the Japanese by M. Reid},
      review={\MR{879273}},
}

\bib{Maxwell}{article}{
      author={Maxwell, James~Clerk},
       title={On the calculation of the equilibrium and stiffness of frames},
        date={1864},
     journal={Philos. Mag.},
      volume={27},
       pages={294\ndash –299},
}

\bib{MLRH14}{inproceedings}{
      author={Muralidhar, Subramanian},
      author={Lloyd, Wyatt},
      author={Roy, Sabyasachi},
      author={Hill, Cory},
      author={Lin, Ernest},
      author={Liu, Weiwen},
      author={Pan, Satadru},
      author={Shankar, Shiva},
      author={Sivakumar, Viswanath},
      author={Tang, Linpeng},
      author={Kumar, Sanjeev},
       title={f4: {F}acebook's warm {BLOB} storage system},
        date={2014},
   booktitle={11th {USENIX} {S}ymposium on {O}perating {S}ystems {D}esign and
  {I}mplementation {(OSDI)}},
       pages={383\ndash 398},
}

\bib{Oxley}{book}{
      author={Oxley, James},
       title={Matroid theory},
     edition={Second},
      series={Oxford Graduate Texts in Mathematics},
   publisher={Oxford University Press, Oxford},
        date={2011},
      volume={21},
        ISBN={978-0-19-960339-8},
         url={https://doi.org/10.1093/acprof:oso/9780198566946.001.0001},
      review={\MR{2849819}},
}

\bib{PollaczekGeiringer}{article}{
      author={Pollaczek-Geiringer, Hilda},
       title={\"{U}ber die gliederung ebener fachwerke.},
        date={1927},
     journal={Z.Angew. Math. Mech.},
      volume={7},
       pages={58\ndash 72},
}

\bib{roth2022higher}{article}{
      author={Roth, Ron~M.},
       title={Higher-order {MDS} codes},
        date={2022},
        ISSN={0018-9448,1557-9654},
     journal={IEEE Trans. Inform. Theory},
      volume={68},
      number={12},
       pages={7798\ndash 7816},
      review={\MR{4544917}},
}

\bib{ron2024efficient}{article}{
      author={Ron-Zewi, Noga},
      author={Venkitesh, S},
      author={Wootters, Mary},
       title={Efficient list-decoding of polynomial ideal codes with optimal
  list size},
        date={2024},
     journal={ar{X}iv:2401.14517},
}

\bib{SingerCucuringu}{article}{
      author={Singer, Amit},
      author={Cucuringu, Mihai},
       title={Uniqueness of low-rank matrix completion by rigidity theory},
        date={2009/10},
        ISSN={0895-4798,1095-7162},
     journal={SIAM J. Matrix Anal. Appl.},
      volume={31},
      number={4},
       pages={1621\ndash 1641},
         url={https://doi.org/10.1137/090750688},
      review={\MR{2595541}},
}

\bib{Sch80}{article}{
      author={Schwartz, J.~T.},
       title={Fast probabilistic algorithms for verification of polynomial
  identities},
        date={1980},
        ISSN={0004-5411,1557-735X},
     journal={J. Assoc. Comput. Mach.},
      volume={27},
      number={4},
       pages={701\ndash 717},
         url={https://doi.org/10.1145/322217.322225},
      review={\MR{594695}},
}

\bib{shivakrishna2022properties}{inproceedings}{
      author={Shivakrishna, D},
      author={Lalitha, V},
       title={Properties of maximally recoverable product codes and higher
  order {MDS} codes},
organization={IEEE},
        date={2022},
   booktitle={2022 {N}ational {C}onference on {C}ommunications {(NCC)}},
       pages={233\ndash 238},
}

\bib{Shivakrishna_MRproduct}{inproceedings}{
      author={Shivakrishna, D.},
      author={Rameshwar, V.~Arvind},
      author={Lalitha, V.},
      author={Sasidharan, Birenjith},
       title={On maximally recoverable codes for product topologies},
        date={2018},
   booktitle={2018 {T}wenty {F}ourth {N}ational {C}onference on
  {C}ommunications {(NCC)}},
       pages={1\ndash 6},
}

\bib{shangguan2023generalized}{article}{
      author={Shangguan, Chong},
      author={Tamo, Itzhak},
       title={Generalized singleton bound and list-decoding {R}eed-{S}olomon
  codes beyond the {J}ohnson radius},
        date={2023},
        ISSN={0097-5397,1095-7111},
     journal={SIAM J. Comput.},
      volume={52},
      number={3},
       pages={684\ndash 717},
         url={https://doi.org/10.1137/20M138795X},
      review={\MR{4591694}},
}

\bib{Tsakiris}{article}{
      author={Tsakiris, Manolis},
       title={Results on the algebraic matroid of the determinantal variety},
        date={2024},
        ISSN={0002-9947,1088-6850},
     journal={Trans. Amer. Math. Soc.},
      volume={377},
      number={1},
       pages={731\ndash 751},
         url={https://doi.org/10.1090/tran/9055},
      review={\MR{4684605}},
}

\bib{whiteley1989matroid}{article}{
      author={Whiteley, Walter},
       title={A matroid on hypergraphs, with applications in scene analysis and
  geometry},
        date={1989},
        ISSN={0179-5376,1432-0444},
     journal={Discrete Comput. Geom.},
      volume={4},
      number={1},
       pages={75\ndash 95},
         url={https://doi.org/10.1007/BF02187716},
      review={\MR{964145}},
}

\bib{AbstractRigidity}{incollection}{
      author={Whiteley, Walter},
       title={Some matroids from discrete applied geometry},
        date={1996},
   booktitle={Matroid theory ({S}eattle, {WA}, 1995)},
      series={Contemp. Math.},
      volume={197},
   publisher={Amer. Math. Soc., Providence, RI},
       pages={171\ndash 311},
         url={https://doi.org/10.1090/conm/197/02540},
      review={\MR{1411692}},
}

\bib{yildiz2019gmmds}{article}{
      author={Yildiz, Hikmet},
      author={Hassibi, Babak},
       title={Optimum linear codes with support-constrained generator matrices
  over small fields},
        date={2019},
        ISSN={0018-9448,1557-9654},
     journal={IEEE Trans. Inform. Theory},
      volume={65},
      number={12},
       pages={7868\ndash 7875},
         url={https://doi.org/10.1109/TIT.2019.2932663},
      review={\MR{4038905}},
}

\bib{Zip79}{incollection}{
      author={Zippel, Richard},
       title={Probabilistic algorithms for sparse polynomials},
        date={1979},
   booktitle={Symbolic and algebraic computation ({EUROSAM} '79, {I}nternat.
  {S}ympos., {M}arseille, 1979)},
      series={Lecture Notes in Comput. Sci.},
      volume={72},
   publisher={Springer, Berlin-New York},
       pages={216\ndash 226},
      review={\MR{575692}},
}

\end{biblist}
\end{bibdiv}

\end{document}